\documentclass[11pt,reqno]{amsart}
\usepackage[english]{babel}
\usepackage[latin1]{inputenc}
\usepackage{fancyhdr}
\usepackage{indentfirst}
\usepackage{graphicx}
\usepackage{float}
\usepackage{caption}
\usepackage{newlfont} 
\usepackage{amssymb}
\usepackage{amsmath}
\usepackage{mathtools}
\usepackage{latexsym}
\usepackage{amsthm}
\usepackage{enumitem}
\usepackage{esint}
\usepackage{amsfonts}
\usepackage{amscd,amsxtra}

\usepackage[english]{babel}
\usepackage[latin1]{inputenc}
\usepackage{fancyhdr}
\usepackage{indentfirst}
\usepackage{graphicx}
\usepackage{float}
\usepackage{caption}
\usepackage{newlfont} 
\usepackage{amssymb}
\usepackage{amsmath}
\usepackage{mathtools}
\usepackage{latexsym}
\usepackage{amsthm}
\usepackage{enumitem}
\usepackage{amsfonts}
\usepackage{amscd,amsxtra}

\newtheorem{thm}{Theorem}[section]
\newtheorem{cor}[thm]{Corollary}
\newtheorem{lem}[thm]{Lemma}
\newtheorem{prop}[thm]{Proposition}
\theoremstyle{definition}

\theoremstyle{remark}
\newtheorem{rem}[thm]{Remark}
\numberwithin{equation}{section}

\DeclareMathOperator*{\diver}{div}

\makeatletter
\def\cleardoublepage{\clearpage\if@twoside \ifodd\c@page\else
	\hbox{}
	\thispagestyle{empty}
	\newpage
	\if@twocolumn\hbox{}\newpage\fi\fi\fi}
\makeatother
\title{A counterexample to the monotone increasing behavior of an Alt-Caffarelli-Friedman formula in the Heisenberg group}
\author{Fausto Ferrari}
\address{Fausto Ferrari: Dipartimento di Matematica\\ Universit\`a di Bologna\\ Piazza di Porta S.Donato 5\\ 40126, Bologna-Italy}
\email{fausto.ferrari@unibo.it }
\author{Nicol\`o Forcillo}
\address{Nicol\`o Forcillo: Dipartimento di Matematica\\ Universit\`a di Bologna\\ Piazza di Porta S.Donato 5\\ 40126, Bologna-Italy}
\email{nicolo.forcillo2@unibo.it }
\thanks{F.F.  and N.F. are partially supported by INDAM-GNAMPA-2019 project: {\it Propriet\`a di regolarit\`a delle soluzioni viscose con applicazioni a problemi di frontiera libera.}}
\date{\today}
\begin{document}
	\begin{abstract}
		In this paper we provide a counterexample about the existence of an increasing monotonicity behavior of a function introduced in \cite{FeFo}, companion of the celebrated Alt-Caffarelli-Friedman monotonicity formula, in the noncommutative framework.
	\end{abstract}
	\maketitle
	
    \tableofcontents	
	\section{Introduction}
	In this paper we continue the research about the existence of a monotonicity formula in the Heisenberg group started in \cite{FeFo}, see also \cite{Forcillo} and \cite{FeFo0}. More precisely, we prove that there exists a function $u$ such that, denoting $u^+:=\sup\{u,0\}$ and $u^{-}:=\sup\{-u,0\}$  defined in a neighborhood of $0\in \mathbb{H}^1,$  if $\Delta_{\mathbb{H}^1}u^\pm\geq 0$   and $u(0)=0,$ then 
	\begin{equation}\label{moneqH}
		J_{u}^{\mathbb{H}^1}(r)\coloneqq \frac{1}{r^4}\int_{B^{\mathbb{H}^1}_r(0)}\frac{|\nabla_{\mathbb{H}^1} u^+(\xi)|^2}{|\xi|_{\mathbb{H}^1}^{2}}\int_{B^{\mathbb{H}^1}_r(0)}\frac{|\nabla_{\mathbb{H}^1} u^-(\xi)|^2}{|\xi|_{\mathbb{H}^1}^{2}}\hspace{0.05cm}d\xi,
	\end{equation}
	$ \xi=(x,y,t)\in \mathbb{H}^1,$ is not monotone increasing in a possibly small right neighborhood of $0,$ where $\mathbb{H}^1$ is the first Heisenberg group.  
	
	In order to better understand the profile of this result, we recall that in \cite{ACF} and \cite{C3} the celebrated monotonicity formula, in the Euclidean setting, was applied to prove regularity results about viscosity solutions of two phase problems like
	\begin{equation}\label{fbpa}
		\left\{
		\begin{array}{l}
			\Delta u=0\:\:\mbox{in}\:\: A^+(u):=\{x\in A:\:u(x)>0\},\\
			\Delta u=0\:\:\mbox{in}\:\: A^-(u):=\mbox{Int}(\{x\in A:\:u(x)\leq 0\}),\\
			|\nabla u^+|^2-|\nabla u^-|^2=1\:\:\mbox{on}\:\:\mathcal{F}(u):=\partial A^+(u)\cap A,
		\end{array}
		\right.
	\end{equation} 	
	where $A \subset \mathbb{R}^n$ is an open set and $u\in C(A)$ is a viscosity solution, see \cite{C1} or \cite{CS} for such definition.
	
	In particular, in \cite{ACF} the authors proved that for every solution $u\in H^1(A)$ of \eqref{fbpa} and for every $P_0\in \mathcal{F}(u),$  the function  
	\begin{equation}\label{monotone_Euclid}
	J_u(r)\coloneqq \frac{1}{r^4}\int_{B_r(P_0)}\frac{|\nabla u^+(P)|^2}{|P-P_0|^{n-2}}dP\int_{B_r(P_0)}\frac{|\nabla u^-(P)|^2}{|P-P_0|^{n-2}}dP
	\end{equation}
	is monotone increasing in a right neighborhood of $0.$ Such tool has been widely employed for proving  regularity results of the solutions of \eqref{fbpa}. About this, we recall \cite{DFSsurvay} for an overview concerning recent results on two-phase problems in the Euclidean framework, see also \cite{CS}. 
	
	As a consequence, in order to study a two-phase problem in a noncommutative group, like the Heisenberg one, it would be useful to have the companion monotonicity formula to \eqref{monotone_Euclid} in this framework.  Indeed, in \cite{DF} the two-phase problem analogous to \eqref{fbpa} has been settled down, i.e. 
	\begin{equation}\label{two_phase_Heisenberg_2}
		\begin{cases}
			\Delta_{\mathbb{H}^n} u=0& \mbox{in }\Omega^+(u):= \{x\in \Omega:\hspace{0.1cm} u(x)>0\},\\
			\Delta_{\mathbb{H}^n} u=0& \mbox{in }\Omega^-(u):=\mbox{Int}(\{x\in \Omega:\hspace{0.1cm} u(x)\leq 0\}),\\
			|\nabla_{\mathbb{H}^n} u^+|^2-|\nabla_{\mathbb{H}^n} u^-|^2= 1&\mbox{on }\mathcal{F}(u):=\partial \Omega^+(u)\cap \Omega,
		\end{cases}
	\end{equation}
	
	where $\Omega\subset \mathbb{H}^n$ and $\Delta_{\mathbb{H}^n},$  $\nabla_{\mathbb{H}^n}$ denote the Kohn-Laplace operator and the horizontal gradient respectively. We refer to Section \ref{Heisenberg_setting} for the main notation and definitions about this noncommutative structure.

	In \cite{FeFo} we proved that if an intrinsic monotone increasing formula exists in $\mathbb{H}^1$, then it should be like \eqref{moneqH}. More precisely, in that paper, 
	supposing $\beta >0,$ we denoted by
	\begin{equation}\label{moneqHold}
		J^{\mathbb{H}^n}_{\beta,u}(r)\coloneqq \frac{1}{r^\beta}\int_{B^{\mathbb{H}^n}_r(0)}\frac{|\nabla_{\mathbb{H}^n} u_1(\zeta)|^2}{|\zeta|_{\mathbb{H}^n}^{Q-2}}\int_{B^{\mathbb{H}^n}_r(0)}\frac{|\nabla_{\mathbb{H}^n} u_2(\zeta)|^2}{|\zeta|_{\mathbb{H}^n}^{Q-2}}\hspace{0.05cm}d\zeta
	\end{equation}
	and we looked for $\beta$ such that $J^{\mathbb{H}^n}_{\beta,u}$ is monotone increasing. Here, $Q:=2n+2$ is the homogeneous dimension of $\mathbb{H}^n.$ For the sake of simplicity, we considered the simplest case only, given by $n=1.$ Hence, we worked on
	\begin{equation}\label{moneqHold1}
		J^{\mathbb{H}^1}_{\beta, u}(r)\coloneqq \frac{1}{r^\beta}\int_{B^{\mathbb{H}^1}_r(0)}\frac{|\nabla_{\mathbb{H}^1} u_1(\xi)|^2}{|\xi|_{\mathbb{H}^1}^{2}}\int_{B^{\mathbb{H}^1}_r(0)}\frac{|\nabla_{\mathbb{H}^1} u_2(\xi)|^2}{|\xi|_{\mathbb{H}^1}^{2}}\hspace{0.05cm}d\xi
	\end{equation}
	only.
	
	On the other hand, since the function $u(x,y,t)=\alpha_1 x^+-\alpha_2 x^-,$  for some fixed numbers $\alpha_i\geq 0,$ $i=1,2,$ satisfies $\Delta_{\mathbb{H}^1}u=0$ in $\{u>0\},$ as well as $\Delta_{\mathbb{H}^1}u=0$ in $\{u<0\},$ we checked that,  if $\beta=4,$
	then $J^{\mathbb{H}^1}_u(r)$ is constant. 
Hence, supposing that $\beta$ is constant and assuming that  $J^{\mathbb{H}^1}_u$ is monotone increasing as well, then necessarily $\beta\leq 4.$

We did not manage to conclude that $J^{\mathbb{H}^1}_u$ is monotone increasing for all the admissible functions, because following the strategy described in \cite{CS}, we needed of some sharp results in geometric measure theory that, in the Heisenberg group, are not known yet, see \cite{FeFo} for the details.  	

	As a consequence, in order to deepen that research, we decided to follow another strategy already available  in the Euclidean case $\mathbb{R}^n.$  More precisely, for selecting the right exponent $\beta,$  and possibly deducing the increasing monotone behavior as well of $J_u$ in $\mathbb{R}^n,$ it is useful to study the behavior of the function
	\begin{equation}\label{functional-harm-funct0}
		\mathcal{I}_{u}(r)\coloneqq \frac{1}{r^2}\int_{B_r(0)}\frac{|\nabla u(P)|^2}{|P|^{n-2}}\hspace{0.05cm}dP
	\end{equation}
when $u$ is harmonic.

	We recall in Section \ref{Euclidean_case}  this last approach already used in \cite{PSU} to introduce the Alt-Caffarelli-Friedman monotone formula. Hence, we decided to follow that parallel proof adapting it to the Heisenberg group.

Nevertheless, on contrary to what we supposed, we discover that  there exists at least a function $u$ such that $\Delta_{\mathbb{H}^1}u=0$ and
	
	\begin{equation}\label{functional-harm-funct0}
		\mathcal{I}_{u}^{\mathbb{H}^1}(r)\coloneqq \frac{1}{r^2}\int_{B_r^{\mathbb{H}^1}(0)}\frac{|\nabla_{\mathbb{H}^1} u(\xi)|^2}{|\xi|^{2}_{\mathbb{H}^1}} \hspace{0.05cm}d\xi 
	\end{equation}
	
	is strictly monotone decreasing in a small right neighborhood of $0,$ differently from what it happens in the Euclidean case.
	
	Hence, starting from this result, we obtain that  $J^{\mathbb{H}^1}_{u}$ is strictly monotone decreasing for a careful choice of $u$. More precisely,  $J^{\mathbb{H}^1}_{u}$ cannot be monotone increasing for all the admissible functions, since the behavior of $J^{\mathbb{H}^1}_{u}$ depends on the choice of $u$ itself. In particular, if $u=x$ we have that $J^{\mathbb{H}^1}_{u}$ is constant, while if 
	$$
	u=x-3yt-2x^3,
	$$
	which satisfies $\Delta_{\mathbb{H}^1}u=0,$ then $J^{\mathbb{H}^1}_{u}$ is strictly monotone decreasing. 
This fact depends on the lack of orthogonality of the intrinsic harmonic polynomials in the Heisenberg group. 

Our main result, whose proof is contained in Section \ref{decreasing_formula}, is the following one.
	\begin{thm}\label{maintheorem}
		Let $u=x-3yt-2x^3.$ Then $\Delta_{\mathbb{H}^1}u=0$ and $J^{\mathbb{H}^1}_u(r)$ is strictly monotone decreasing in a right neighborhood of $r=0.$ 
	\end{thm}	
	In Section \ref{explicit_computation} we show the explicit computation	of the fact that $\mathcal{I}^{\mathbb{H}^1}_u$ is strictly monotone decreasing, obtaining the main tool for proving Theorem \ref{maintheorem}. In Section \ref{gen_app} we provide an extension of our argument and an application exhibiting a genuine nontrivial example of solution to a two-phase  free boundary problem in the Heisenberg group. We conclude this introduction pointing out that in Section \ref{harmonic_properties} we study the behavior of harmonic polynomials in order to justify the way we obtained the specific counterexample we need to. In fact, independently to our application, we think  that the result contains therein may be interesting and useful by itself, especially for further applications. For instance in case we want to determine other polynomials with special properties with respect to $\Delta_{\mathbb{H}^1}$.
	\section{Preliminary facts: the Euclidean case}	\label{Euclidean_case}
	In this section, we  provide a preliminary fact 
	inspired by the argument contained in \cite{PSU} that here below we recall in detail.
	
	%
	
	More precisely, given a harmonic function $u$ in an open set $A\subset \mathbb{R}^n,$ we consider the following function depending on the radius $r$ of the Euclidean ball $B_r(0)\subset A:$ 
	\begin{equation}\label{functional-harm-funct}
		\mathcal{I}_{u}(r)\coloneqq \frac{1}{r^2}\int_{B_r(0)}\frac{|\nabla u(P)|^2}{|P|^{n-2}}\hspace{0.05cm}dP.
	\end{equation}
	Here $|P|$ denotes the usual Euclidean norm of $P\in A\subset \mathbb{R}^n,$ as well as $|\nabla u(P)|$ is the usual Euclidean norm of the Euclidean gradient of the function $u.$
	
	We want to rewrite \eqref{functional-harm-funct} in a more convenient way, for understanding its behavior when $r\to0.$ To this end, we write down the Taylor expansion of $u$ at $0$
	\[u(x)=\sum_{\alpha}a_{\alpha}x^{\alpha},\]
	where $x=(x_n,\dots,x_n)\in \mathbb{R}^n,$
	$\alpha\coloneqq (\alpha_1,\ldots,\alpha_n)\in (\mathbb{N}\cup\{0\})^n$ is a multi-index, $|\alpha|=\sum_{i=1}^n\alpha_i,$ $x^{\alpha}:=x_1^{\alpha_1}\cdots x_n^{\alpha_n}.$

	In particular, in our case, $u$ is harmonic. Hence we can group together the terms with the same homogeneity to achieve
	\begin{equation}\label{expansion-in-homog-polyn-harm-funct}
		u(x)=\sum_{|\alpha|=k=0}^{\infty}P_{k}(x),
	\end{equation}
	where $P_{k}(x)$ is a homogeneous polynomial of degree $k.$ On the other hand, being $u$ harmonic, each $P_k(x)$ has to be harmonic as well. As a consequence, we have that $u$ is a sum (possibly infinite) of homogeneous harmonic polynomials.

	Now, if we substitute \eqref{expansion-in-homog-polyn-harm-funct} into $\mathcal{I}_u,$ see \eqref{functional-harm-funct}, we have to compute the gradient of the homogeneous polynomials $P_{k}.$ For this purpose, we collect this calculation in the following technical lemma.
	\begin{lem}\label{lemma-grad-homog-polyn}
		Let 
		\[P_{k}(x)=\sum_{\beta,\hspace{0.1cm}|\beta|=k}b_{\beta}x^{\beta}\]
		be a homogeneous polynomial of degree $k.$ Then 
		\[\nabla P_{k}(x)=\sum_{\beta,\hspace{0.1cm}|\beta|=k}b_{\beta}(\beta_1 x_1^{\beta_1-1}x_2^{\beta_2}\cdots x_n^{\beta_n},\ldots,\beta_n x_1^{\beta_1}x_2^{\beta_2}\cdots x_n^{\beta_n-1}).\] 
	\end{lem}
	The proof is a straightforward computation.\\
	Therefore, we are ready to rewrite $\mathcal{I}_u$ using \eqref{expansion-in-homog-polyn-harm-funct} and Lemma \ref{lemma-grad-homog-polyn} as well. Hence, from \eqref{expansion-in-homog-polyn-harm-funct} we have
	\begin{equation}\label{rewriting-J_u-1}
		\mathcal{I}_{u}(r)=\frac{1}{r^2}\int_{B_r}\frac{|\nabla( \sum_{k=0}^{\infty}P_{k}(x))|^2}{|x|^{n-2}}\hspace{0.05cm}dx=\frac{1}{r^2}\int_{B_r}\frac{|\sum_{k=1}^{\infty}\nabla P_{k}(x)|^2}{|x|^{n-2}}\hspace{0.05cm}dx.
	\end{equation}
	We apply the classical coarea formula, see \cite{Federer}, and performing the change of variables in spherical coordinates to get
	\begin{align*}
		&\mathcal{I}_{u}(r)=\frac{1}{r^2}\int^r_0\bigg(\int_{\partial B_t}\frac{|\sum_{k=1}^{\infty}(\nabla P_{k})(\eta)|^2}{t^{n-2}}\hspace{0.05cm}d\eta\bigg)dt\\
		&=\frac{1}{r^2}\int^r_0\bigg(\int_{\partial B_1}\frac{|\sum_{k=1}^{\infty}(\nabla P_{k})(t\sigma)|^2}{t^{n-2}}t^{n-1}\hspace{0.05cm}d\sigma\bigg)dt\\
		&=\frac{1}{r^2}\int^r_0t\bigg(\int_{\partial B_1}\bigg(\sum_{k=1}^{\infty}|(\nabla P_{k})(t\sigma)|^2\bigg)d\sigma\bigg)dt\\
		&+\frac{1}{r^2}\int^r_0t\bigg(\int_{\partial B_1}\bigg(\sum_{\stackrel{h,k=1}{h\neq k}}^{\infty}\langle(\nabla P_{h})(t\sigma),(\nabla P_{k})(t\sigma)\rangle\bigg)d\sigma\bigg)dt,
	\end{align*}
	which yields
	\begin{equation}\label{rewriting-J_u-2}
		\begin{split}
			&\mathcal{I}_{u}(r)=\frac{1}{r^2}\int^r_0t\bigg(\sum_{k=1}^{\infty}\int_{\partial B_1}|(\nabla P_{k})(t\sigma)|^2d\sigma\bigg)dt\\
			&+\frac{1}{r^2}\int^r_0t\bigg(\sum_{\stackrel{h,k=1}{h\neq k}}^{\infty}\int_{\partial B_1}\langle(\nabla P_{h})(t\sigma),(\nabla P_{k})(t\sigma)\rangle\hspace{0.05cm}d\sigma\bigg)dt.
		\end{split}
	\end{equation}
	Now, exploiting Lemma \ref{lemma-grad-homog-polyn}, we point out that $\nabla P_k$ has each component which is a homogeneous polynomial of degree $k-1.$ Thus, it holds
	\[(\nabla P_{k})(t\sigma)=t^{k-1}(\nabla P_{k})(\sigma),\]
	which entails, according to \eqref{rewriting-J_u-2},
	\begin{align*}
		&\mathcal{I}_{u}(r)=\frac{1}{r^2}\int^r_0t\bigg(\sum_{k=1}^{\infty}t^{2k-2}\int_{\partial B_1}|(\nabla P_{k})(\sigma)|^2d\sigma\bigg)dt\\
		&+\frac{1}{r^2}\int^r_0t\bigg(\sum_{\stackrel{h,k=1}{h\neq k}}^{\infty}t^{h+k-2}\int_{\partial B_1}\langle(\nabla P_{h})(\sigma),(\nabla P_{k})(\sigma)\rangle\hspace{0.05cm}d\sigma\bigg)dt\\
		&=\frac{1}{r^2}\int^r_0\bigg(\sum_{k=1}^{\infty}t^{2k-1}\int_{\partial B_1}|(\nabla P_{k})(\sigma)|^2\bigg)d\sigma\hspace{0.05cm}dt\\
		&=\frac{1}{r^2}\sum_{k=1}^{\infty}\int_{\partial B_1}|(\nabla P_{k})(\sigma)|^2\bigg(\int^r_0t^{2k-1}\hspace{0.05cm}dt\bigg)d\sigma\\
		&=\frac{1}{r^2}\sum_{k=1}^{\infty}\int_{\partial B_1}|(\nabla P_{k})(\sigma)|^2\frac{r^{2k}}{2k}d\sigma=\sum_{k=1}^{\infty}a_kr^{2(k-1)},
	\end{align*}
	where
	\begin{equation}\label{defin-a_k-funct-J_u}
		a_k\coloneqq \frac{1}{2k}\int_{\partial B_1}|(\nabla P_{k})(\sigma)|^2\hspace{0.05cm}d\sigma,
	\end{equation}
	and and having used the fact that
	\begin{equation}\label{orthogonality-grad-homog-harm-polyn}
		\int_{\partial B_1}\langle(\nabla P_{h})(\sigma),(\nabla P_{k})(\sigma)\rangle\hspace{0.05cm}d\sigma=0,\quad h\neq k.
	\end{equation}
	To recap, we have showed that
	\begin{equation}\label{rewriting-J_u-final}
		\mathcal{I}_{u}(r)=\sum_{k=1}^{\infty}a_kr^{2(k-1)}.
	\end{equation}
	Hence, we immediately obtain the following result.
	\begin{prop}\label{prop-limit-r-tend-to-0-monotonic-J_u}
		Let $\mathcal{I}_u$ be as in \eqref{functional-harm-funct}. Then $\mathcal{I}_u$ is monotone increasing with respect to $r$ and $\lim\limits_{r\rightarrow 0}\mathcal{I}_u(r)=a_1,$ with $a_1$ defined in \eqref{defin-a_k-funct-J_u}.
	\end{prop}
	\begin{proof}
		The proof follows by two straightforward computations, using \eqref{rewriting-J_u-final}. Precisely, about the first one, letting $r$ tend to $0$ in \eqref{rewriting-J_u-final}, we obtain
		\[\lim\limits_{r\rightarrow 0}\mathcal{I}_{u}(r)=\lim\limits_{r\rightarrow 0}\bigg(\sum_{k=1}^{\infty}a_kr^{2(k-1)}\bigg)=a_1.\]
		Concerning the second one, if we compute the first derivative of $\mathcal{I}_u$ directly from \eqref{rewriting-J_u-final}, we get
		\[\mathcal{I}_{u}'(r)=\bigg(\sum_{k=1}^{\infty}a_kr^{2(k-1)}\bigg)'=\sum_{k=2}^{\infty}2(k-1)a_kr^{2k-3}\ge 0\]
		because the $a_k$'s are nonnegative by virtue of \eqref{defin-a_k-funct-J_u}, and hence $\mathcal{I}_{u}$ is indeed monotone increasing.	
	\end{proof}
	Let us discuss that $\lim\limits_{r\rightarrow 0}\mathcal{I}_{u}(r)=a_1.$ By definition, we have 
	\[a_1=\frac{1}{2}\int_{\partial B_1}|\nabla P_1(\sigma)|^2d\sigma.\]
	We then recall that $P_1$ is the homogeneous polynomial of degree $1$ coming from the Taylor expansion of $u$ at $0.$ As a consequence, according to Lemma \ref{lemma-grad-homog-polyn}, $\nabla P_1$ is a constant. Specifically, it is the gradient of $u$ at $0.$ So, it follows
	\[a_1=\frac{|\nabla u(0)|^2|B_1|}{2}.\]
	This equality tells us that the limit of $\mathcal{I}_u$ with $r$ tending to $0$ is strictly positive depending on whether the gradient of $u$ vanishes in $0$ or not.\\
	To conclude this argument in the Euclidean case, we show a proof of \eqref{orthogonality-grad-homog-harm-polyn}. Precisely, we recall the following result, see \cite{Mul}.
	\begin{lem}\label{lemma-orthogonality-homog-harm-polyn}
		Let $P_h$ and $P_k$ be two harmonic homogeneous polynomials in $B_1$ of degree $h$ and $k$ respectively, with $h\neq k.$ Then, it holds
		\begin{equation}\label{condition-orthogon-homog-harm-polyn}
			\int_{\partial B_1}P_hP_k\hspace{0.1cm}d\sigma=0.
		\end{equation}
	\end{lem}
	\begin{proof}
		The proof is a direct consequence of the divergence theorem in the Euclidean setting and Lemma \ref{lemma-grad-homog-polyn}. Indeed, since $P_h$ and $P_k$ are both harmonic in $B_1,$ we achieve from the divergence theorem
		\begin{equation}\label{orthogonality-harm-homog-polyn-diff-degrees-1}
			\begin{split}
				&0=\int_{B_1}(P_h\Delta P_k-P_k\Delta P_h)\hspace{0.05cm}dx=\int_{B_1}\diver(P_h\nabla P_k-P_k\nabla P_h)\hspace{0.05cm}dx\\
				&=\int_{\partial B_1}\langle P_h\nabla P_k-P_k\nabla P_h,x\rangle\hspace{0.05cm}d\sigma(x)=\int_{\partial B_1}(P_h\langle\nabla P_k,x\rangle-P_k\langle\nabla P_h,x\rangle)\hspace{0.05cm}d\sigma(x),
			\end{split}
		\end{equation}
		because $x$ is the outward normal to $\partial B_1.$\\
		Now, we can exploit Lemma \ref{lemma-grad-homog-polyn} to get an expression of $\langle\nabla P_k,x\rangle$ for any $k.$ Specifically, we have
		\begin{align}\label{scalar-product-grad-homog-polyn-x}
			&\langle\nabla P_k,x\rangle=\sum_{\beta,\hspace{0.1cm}|\beta|=k}b_{\beta}(\beta_1x_1^{\beta_1}x_2^{\beta_2}\cdots x_n^{\beta_n}+\ldots+\beta_n x_1^{\beta_1}x_2^{\beta_2}\cdots x_n^{\beta_n})\\
			&=\sum_{\beta,\hspace{0.1cm}|\beta|=k}b_{\beta}(\beta_1+\cdots+\beta_n)x_1^{\beta_1}x_2^{\beta_2}\cdots x_n^{\beta_n}=\sum_{\beta,\hspace{0.1cm}|\beta|=k}b_{\beta}|\beta|x^{\beta}=kP_k.
		\end{align}
		Then, exploiting \eqref{scalar-product-grad-homog-polyn-x}, we obtain from \eqref{orthogonality-harm-homog-polyn-diff-degrees-1}
		\begin{align*}
			&0=\int_{\partial B_1}(P_h\langle\nabla P_k,x\rangle-P_k\langle\nabla P_h,x\rangle)\hspace{0.05cm}d\sigma(x)=\int_{\partial B_1}(P_hkP_k-P_kh P_h)\hspace{0.05cm}d\sigma(x),
		\end{align*}
		which gives
		\[\int_{\partial B_1}P_hP_k\hspace{0.05cm}d\sigma=0,\]
		because $k\neq h.$
	\end{proof}
	It is now an immediate consequence to achieve \eqref{orthogonality-grad-homog-harm-polyn}, simply observing that, according to Lemma \ref{lemma-grad-homog-polyn}, each component of the gradient of a homogeneous polynomial is still a homogeneous polynomial.
	\section{The Heisenberg setting}\label{Heisenberg_setting}

	Following the idea introduced in Section \ref{Euclidean_case}, we face the same problem in the framework of $\mathbb{H}^1.$ 
	 For the sake of simplicity, we restrict ourselves to the $\mathbb{H}^1$ case, nevertheless the argument holds in $\mathbb{H}^n$ as well. 
	
	We recall here that
	$\mathbb{H}^n$ denotes the set $\mathbb{R}^{2n+1},$ $n\in \mathbb{N},$ $n\geq 1,$ 
	endowed with the noncommutative inner law in such  a way that for every $P\equiv (x_1,y_1,t_1)\in \mathbb{R}^{2n+1},$ $M\equiv (x_2,y_2,t_2)\in \mathbb{R}^{2n+1},$ $x_{i}\in \mathbb{R}^n,$ $y_{i}\in \mathbb{R}^n,$ $i=1,2$ it holds: 
	$$
	P\circ M:=(x_1+x_2,y_1+y_2,t_1+t_2+2(\langle x_2, y_1\rangle- \langle x_1, y_2\rangle)),
	$$
	where $\langle \cdot, \cdot\rangle$ denotes the usual inner product in $\mathbb{R}^n.$

 Let $X_i=(e_i,0,2y_i)$ and $Y_i=(0,e_i,-2x_i),$ $i=1,\dots,n,$ where $\{e_i\}_{1\leq i\leq n}$ is the canonical basis for $\mathbb{R}^n.$ 
The inverse of $P:=(x,y,t)\not=0$ is $(-x,-y,-t)$ and it is denoted by $P^{-1}.$
	
	We use the same symbols to denote the vector fields associated with the previous vectors, so that for $i=1,\dots,n,$ we have:
	$$
	X_i:=\partial_{x_i}+2y_i\partial_t,\quad
	Y_i:=\partial_{y_i}-2x_i\partial_t.
	$$
	The commutator between the vector fields is
	$$
	[X_i,Y_i]:=X_iY_i-Y_iX_1=-4\partial_t,\quad i=1,\ldots,n,
	$$
	otherwise is $0.$ 
	The intrinsic gradient of a real valued smooth function $u$ in a point $P$ is  
	$$
	\nabla_{\mathbb{H}^n}u(P):=\sum_{i=1}^n(X_iu(P)X_i(P)+Y_iu(P)Y_i(P)).
	$$
	
	Now, there exists a unique metric on 
	$$H\mathbb{H}^n_P=\mbox{span}\{X_1(P),\dots,X_n(P),Y_1(P),\dots,Y_n(P)\}$$ which makes orthonormal the set of vectors $\{X_1,\dots,X_n,Y_1,\dots,Y_n\}.$ Thus, for every $P\in \mathbb{H}^n$ and for every $U,W\in H\mathbb{H}^n_P,$ $$U=\sum_{j=1}^n(\alpha_{1,j}X_{j}(P)+\beta_{1,j}Y_j(P)),$$
	$$V=\sum_{j=1}^n(\alpha_{2,j}X_{j}(P)+\beta_{2,j}Y_j(P)),$$ it holds
	$$
	\langle U,V\rangle=\sum_{j=1}^n(\alpha_{1,j}\alpha_{2,j}+\beta_{1,j}\beta_{2,j}).
	$$  
	Since we mainly work on $\mathbb{H}^1,$ that is to the case in which $n=1,$ we simply continue introducing the remaining notation in $\mathbb{H}^1.$
	In particular, we define a norm associated with the metric on the space $\mbox{span}\{X,Y\},$ as it follows:
	$$
	\mid U\mid=\sqrt{\sum_{j=1}^1\left(\alpha_{1,j}^2+\beta_{1,j}^2\right)}=\sqrt{\alpha_{1,1}^2+\beta_{1,1}^2}.
	$$
	For example, the norm of the intrinsic gradient of a smooth function $u$ in $P$ is
	$$
	\mid \nabla_{\mathbb{H}^1} u(P)\mid=\sqrt{(Xu(P))^2+(Yu(P))^2}.
	$$
	Moreover, if $\nabla_{\mathbb{H}^1} u(P)\not=0,$ then
	$$
	\left|\frac{\nabla_{\mathbb{H}^1}u(P)}{\mid \nabla_{\mathbb{H}^1}u(P)\mid}\right|=1.
	$$
	
	If $\nabla_{\mathbb{H}^1} u(P)=0,$ instead, we say that the point $P$ is characteristic for the smooth surface $\{u=u(P)\}.$
	In particular, for every point $M\in \{u=u(P)\},$ which is not characteristic, it is well defined the intrinsic normal to the surface $\{u=u(P)\}$ as follows: 
	$$
	\nu(M)=\frac{\nabla_{\mathbb{H}^1} u(M)}{\mid \nabla_{\mathbb{H}^1} u(M)\mid}.
	$$
	The Kohn-Laplace operator is 
	$$\Delta_{\mathbb{H}^1}=X^2+Y^2,$$ 
	where
	it results
	\begin{equation}
		\Delta_{\mathbb{H}^1}=\frac{\partial^2}{\partial x^2}+\frac{\partial^2}{\partial y^2}+2y\frac{\partial^2}{\partial x\partial t}-2x\frac{\partial^2}{\partial y\partial t}+4(x^2+y^2)\frac{\partial^2}{\partial t^2},
	\end{equation} 
	so that $\Delta_{\mathbb{H}^1}$ is a degenerate elliptic operator because the smallest eigenvalue associated with the matrix
	\begin{equation*}
		\left(
		\begin{array}{ccc}
			1,&0,&2y\\
			0,&1,&-2x\\
			2y,&-2x,&4(x^2+y^2)
		\end{array}\right)
	\end{equation*}
	is always $0.$ 
	
	At this point, we introduce in the Heisenberg group $\mathbb{H}^1$ the Koranyi norm of $P\equiv (x,y,t)\in\mathbb{H}^1$ as
	$$
	| (x,y,t)|_{\mathbb{H}^1}:=\sqrt[4]{( x^2+y^2)^2+t^2}.
	$$
	In particular, for every positive number $r,$ the gauge ball of radius $r$ centered at $0$ is
	$$
	B^{\mathbb{H}^1}_r(0):=\{P\in \mathbb{H}^1 :\:\:|P|_{\mathbb{H}^1}<r\}.
	$$
	It is worth to say that this structure is endowed by suitable properties, like the left invariance with respect to the inner law.
	More precisely, for every point $P\in \mathbb{H}^1$
	\begin{equation}\label{invariant1}
		P\circ B^{\mathbb{H}^1}_r(0)=B^{\mathbb{H}^1}_r(P)=\{S\in \mathbb{H}^1 :\:\:|P^{-1}\circ S|_{\mathbb{H}^1}<r\}
	\end{equation} 
	and for every $P\in \mathbb{H}^1$ it results
	$$
	\mbox{meas}_{3}(B^{\mathbb{H}^1}_r(0))=\mbox{meas}_{3}(B^{\mathbb{H}^1}_r(P)),
	$$
	where $\mbox{meas}_{3}$ denotes the usual Lebesgue measure in $\mathbb{R}^3.$
	
	Moreover, if $u$ is a $C^1$ function in $\mathbb{H}^1$ and for every $P\in \mathbb{H}^1$ if $v(S):=u(P\circ S)$ then
	\begin{equation}\label{invariant2}
		Xv(S)=Xu(P\circ S),\:\:Yv(S)=Yu(P\circ S),\:\: \Delta_{\mathbb{H}^1}v(S)=\Delta_{\mathbb{H}^1}u(P\circ S).
	\end{equation} 
	
	In addition, a dilation semigroup is defined as follows: for every $r>0$ and for every $P\equiv (x,y,t)\in \mathbb{H}^1,$ let
	\begin{equation}\label{dilation}\begin{split}
			\delta_r(P):=(rx,ry,r^2t).
		\end{split}
	\end{equation}
	As a consequence,  denoting $u_r(S):=u(\delta_r(S)),$  it results that
	\begin{equation*}\begin{split}
			Xu_r(S)=r(Xu)(\delta_r(S)),\:\:Yu_r(S)=r(Yu)(\delta_r(S)),
		\end{split}
	\end{equation*}
	and
	\begin{equation*}\begin{split}
			\Delta_{\mathbb{H}^1}u_r(S)=r^2(\Delta_{\mathbb{H}^1}u)(\delta_r(S)).
		\end{split}
	\end{equation*}	
	The details of all previous properties can be found in \cite{Stein}, or in other handbooks like \cite{CDPT} or  \cite{BLU}. 
	Moreover, the following
representation theorem holds, see  \cite{capdangar}, \cite{FSSChouston} and \cite{Magnani} for further developments.

\begin{prop}\label{perimetro regolare}
If $E$ is a $\mathbb{H}^{1}:=\mathbb{R}^{3}$-Caccioppoli set with Euclidean ${\mathbf C}^1$
boundary, then there is an explicit representation of the
$\mathbb{H}^{1}$-perimeter in terms of the Euclidean $2$-dimensional
Hausdorff measure $\mathcal H^{2}$
\begin{equation*}
P_{\mathbb{H}^1}^{\Omega, E}(\partial E)=\int_{\partial
E\cap\Omega}\sqrt{\langle
X,n_E\rangle_{\mathbb{R}^{3}}^2+\langle
Y,n_E\rangle_{\mathbb{R}^{3}}^2}d{\mathcal {H}}^{2},
\end{equation*}
where $n_E=n_E(x)$ is the Euclidean unit outward normal to $\partial
E$.
\end{prop}

We also have the following intrinsic divergence theorem:
\begin{prop}\label{divergence}
If $E$ is a regular bounded open set with Euclidean ${\mathbf C}^1$
boundary and $\phi$ is a horizontal vector field,
continuously differentiable on $ \overline{\Omega} $, then
$$
\int_E \mathrm{div}_{\mathbb{H}^1}\ \phi\, dx = \int_{\partial E} \langle \phi, \nu_{\mathbb{H}^1}\rangle d P_{\mathbb{H}^1}^{E},
$$
where $\nu_{\mathbb{H}^1}(x)$ is the intrinsic horizontal unit outward normal to $\partial E$,
given by the (normalized) projection of $n_E(x)$ on the fiber $H\mathbb{H}^1_x$ of
the horizontal fiber bundle $H\mathbb{H}^1$.
\end{prop}

	\section{Some properties of the orthogonal polynomials in the Heisenberg group}\label{harmonic_properties}
	In this section we develop our approach.\\
	Let us take the companion functional to \eqref{functional-harm-funct} in $\mathbb{H}^n,$ namely, for every $P\in \mathbb{H}^n,$
	\[
	\mathcal{I}_{u}^{\mathbb{H}^n}(r)\coloneqq \frac{1}{r^2}\int_{B^{\mathbb{H}^n}_r(P)}\frac{|\nabla_{\mathbb{H}^n} u|^2}{|P^{-1}\circ\xi|_{\mathbb{H}^n}^{Q-2}}\hspace{0.05cm}d\xi,\quad \xi=(x,y,t)\in \mathbb{H}^n,
	\]
	where $Q$ is $4$ in ${\mathbb{H}^1},$ and thus $\mathcal{I}_{u}^{\mathbb{H}^1}$ becomes
	\begin{equation}\label{functional-harm-funct-H^10}
		\mathcal{I}_{u}^{\mathbb{H}^1}(r)= \frac{1}{r^2}\int_{B^{\mathbb{H}^1}_r(P)}\frac{|\nabla_{\mathbb{H}^1} u|^2}{|P^{-1}\circ\xi|_{\mathbb{H}^1}^{2}}\hspace{0.05cm}d\xi.
	\end{equation}
	On the other hand, recalling the translation invariance properties \eqref{invariant1}, \eqref{invariant2},
	we obtain
	\begin{equation}\label{functional-harm-funct-H^1}
		\mathcal{I}_{u}^{\mathbb{H}^1}(r)= \frac{1}{r^2}\int_{B^{\mathbb{H}^1}_r(0)}\frac{|\nabla_{\mathbb{H}^1} u(P\circ \xi)|^2}{|\xi|_{\mathbb{H}^1}^{2}}\hspace{0.05cm}d\xi=\frac{1}{r^2}\int_{B^{\mathbb{H}^1}_r(0)}\frac{|\nabla_{\mathbb{H}^1} v(\xi)|^2}{|\xi|_{\mathbb{H}^1}^{2}}\hspace{0.05cm}d\xi,
	\end{equation}
	with $v(\xi)=u(P\circ \xi).$ 
	
	Hence, we can suppose to work, without restrictions, in the standard case in which our function $u$ is defined in a neighborhood of $0\in \mathbb{H}^1.$
	
	Following the parallelism with the Euclidean case, we assume that the function $u$ in $\mathcal{I}_{u}^{\mathbb{H}^1}$ is $\mathbb{H}^1$-harmonic, that is 
	$\Delta_{\mathbb{H}^1}u=0,$ 
	in an open set $\Omega \subset \mathbb{H}^1$ such that $0\in \Omega.$
	
	At this point, we want to rewrite \eqref{functional-harm-funct-H^1}, supposing that $r$ is small. Specifically, we have the following equality.
	\begin{prop}
		Let $\mathcal{I}_{u}^{\mathbb{H}^1}$ be defined in \eqref{functional-harm-funct-H^1}. It holds
		\begin{equation}\label{rewriting-J_u^H^1-2}
			\begin{split}
				&\mathcal{I}_{u}^{\mathbb{H}^1}(r)=\frac{1}{r^2}\int^r_0s\bigg(\sum_{k=1}^{\infty}\int_{\partial B^{\mathbb{H}^1}_1(0)}\frac{|(\nabla_{\mathbb{H}^1} P_{k})(\delta_s(\sigma))|^2}{\sqrt{\sigma_x^2+\sigma_y^2}}\hspace{0.05cm}d\sigma_{\mathbb{H}^1}\bigg)ds\\
				&+\frac{1}{r^2}\int^r_0s\bigg(\sum_{\stackrel{h,k=1}{h\neq k}}^{\infty}\int_{\partial B^{\mathbb{H}^1}_1(0)}\frac{\langle(\nabla_{\mathbb{H}^1} P_{h})(\delta_s(\sigma)),(\nabla_{\mathbb{H}^1} P_{k})(\delta_s(\sigma))\rangle}{\sqrt{\sigma_x^2+\sigma_y^2}}\hspace{0.05cm}d\sigma_{\mathbb{H}^1}\bigg)ds,
			\end{split}
		\end{equation}
		with $\sigma=(\sigma_x,\sigma_y,\sigma_t)\in \partial B^{\mathbb{H}^1}_1(0).$
	\end{prop}
	\begin{proof}
		Since \eqref{expansion-in-homog-polyn-harm-funct} holds in the same way for $u,$ even if $k$ is not necessarily the height of $\alpha,$ repeating the same argument used to obtain \eqref{rewriting-J_u-1}, we have from \eqref{functional-harm-funct-H^1}
		\begin{equation}\label{rewriting-J_u^H^1-1}
			\mathcal{I}_{u}^{\mathbb{H}^1}(r)=\frac{1}{r^2}\int_{B^{\mathbb{H}^1}_r(0)}\frac{|\sum_{k=1}^{\infty}\nabla_{\mathbb{H}^1} P_{k}|^2}{|\xi|_{\mathbb{H}^1}^{2}}\hspace{0.05cm}d\xi,
		\end{equation}
		where $P_k$ is homogeneous of degree $k$ in $\mathbb{H}^1,$ see Remark \ref{remark-express-homog-polyn-H^1}.\\
		We now want to exploit the coarea formula and the change of variables in spherical coordinates in $\mathbb{H}^1$ to further rewrite $\mathcal{I}_{u}^{\mathbb{H}^1}.$ Before applying them, we stress that $|\nabla_{\mathbb{H}^1} |\xi|_{\mathbb{H}^1}|\neq 0$ whenever $(x,y)\not=(0,0)\in \mathbb{R}^2.$ In particular, it holds \begin{equation}\label{norm-grad-heisen-koranyi-norm-H^1}
			|\nabla_{\mathbb{H}^1}|\xi|_{\mathbb{H}^1}|=|\xi|_{\mathbb{H}^1}^{-1}\sqrt{x^2+y^2}.
		\end{equation}
		We then achieve, from \eqref{rewriting-J_u^H^1-1}, 
		\begin{align*}
			&\mathcal{I}_{u}^{\mathbb{H}^1}(r)
=\frac{1}{r^2}\int^r_0\bigg(\int_{\partial B^{\mathbb{H}^1}_s(0)}\frac{|\sum_{k=1}^{\infty}(\nabla_{\mathbb{H}^1} P_{k})(\eta)|^2}{s^2\sqrt{\eta_x^2+\eta_y^2}}s\hspace{0.05cm}d\eta_{\mathbb{H}^1}\bigg)ds\\
			&
=\frac{1}{r^2}\int^r_0\bigg(\int_{\partial B^{\mathbb{H}^1}_1(0)}\frac{|\sum_{k=1}^{\infty}(\nabla_{\mathbb{H}^1} P_{k})(\delta_s(\sigma))|^2}{s\sqrt{s^2\sigma_x^2+s^2\sigma_y^2}}\hspace{0.05cm}s^3\hspace{0.05cm}d\sigma_{\mathbb{H}^1}\bigg)ds,
		\end{align*}
		which implies
		\begin{equation*}
			\begin{split}
				&\mathcal{I}_{u}^{\mathbb{H}^1}(r)=\frac{1}{r^2}\int^r_0s\bigg(\sum_{k=1}^{\infty}\int_{\partial B^{\mathbb{H}^1}_1(0)}\frac{|(\nabla_{\mathbb{H}^1} P_{k})(\delta_s(\sigma))|^2}{\sqrt{\sigma_x^2+\sigma_y^2}}\hspace{0.05cm}d\sigma_{\mathbb{H}^1}\bigg)ds\\
				&+\frac{1}{r^2}\int^r_0s\bigg(\sum_{\stackrel{h,k=1}{h\neq k}}^{\infty}\int_{\partial B^{\mathbb{H}^1}_1(0)}\frac{\langle(\nabla_{\mathbb{H}^1} P_{h})(\delta_s(\sigma)),(\nabla_{\mathbb{H}^1} P_{k})(\delta_s(\sigma))\rangle}{\sqrt{\sigma_x^2+\sigma_y^2}}\hspace{0.05cm}d\sigma_{\mathbb{H}^1}\bigg)ds,
			\end{split}
		\end{equation*}
		with $\sigma=(\sigma_x,\sigma_y,\sigma_t)\in \partial B^{\mathbb{H}^1}_1(0).$
	\end{proof}
	\begin{rem}\label{remark-express-homog-polyn-H^1}
		By the definition of the dilation semigroup in $\mathbb{H}^1,$ see \eqref{dilation}, it follows that
		\begin{equation}\label{homogeneity-xi^beta}
			\delta_r(\xi)^{\beta}=(rx)^{\beta_1}(ry)^{\beta_2}(r^2t)^{\beta_3}=r^{\beta_1+\beta_2+2\beta_3}\xi^{\beta}.
		\end{equation}
		This fact then entails that $\xi^{\beta}$ is homogeneous of degree $|\beta|+\beta_3.$ Hence, a homogeneous polynomial $P_k$ of degree $k$ in $\mathbb{H}^1$ has the form
		\begin{equation}\label{express-homog-polyn-H^1}
			P_{k}(\xi)=\sum_{\beta,\hspace{0.1cm}|\beta|+\beta_3=k}b_{\beta}\xi^{\beta}.
		\end{equation}
		Properties of polynomials in the Heisenberg group have been investigated in \cite{Greiner}, where it has been pointed out how the harmonic polynomials in the Heisenberg group have done.
	\end{rem}
	Therefore, to understand the behavior of $\mathcal{I}_{u}^{\mathbb{H}^1}$ it is useful to compute the gradient of homogeneous polynomials.
	\begin{lem}\label{lemma-grad-homog-polyn-H^1}
		Let $P_{k}$ be as in \eqref{express-homog-polyn-H^1}. Then, it holds
		\[\nabla_{\mathbb{H}^1} P_{k}(\xi)=\sum_{\beta,\hspace{0.1cm}|\beta|+\beta_3=k}b_{\beta}\xi^{\beta}(\beta_1 x^{-1}+2\beta_3 yt^{-1},\beta_2 y^{-1}-2\beta_3 xt^{-1}).\]
	\end{lem}
	\begin{proof}
		The proof is a straightforward computation. Given a smooth function $u:\mathbb{H}^1\rightarrow\mathbb{R},$ we denote by $\nabla_{\mathbb{H}^1}u$ as
		\begin{equation}\label{defin-grad-heisen-H^1}
			\nabla_{\mathbb{H}^1}u\equiv (Xu,Yu),\quad X\coloneqq\frac{\partial}{\partial x}+2y\frac{\partial}{\partial t},\quad Y\coloneqq \frac{\partial}{\partial y}-2x\frac{\partial}{\partial t}.
		\end{equation}
		Consequently, we obtain
		\[\nabla_{\mathbb{H}^1} P_{k}(\xi)=\sum_{\beta,\hspace{0.1cm}|\beta|+\beta_3=k}b_{\beta}(X\xi^{\beta},Y\xi^{\beta}).\]
		So, it remains to calculate $X\xi^{\beta}$ and $Y\xi^{\beta}.$ We get by \eqref{defin-grad-heisen-H^1}
		\begin{align*}
			&X\xi^{\beta}=\beta_1 x^{\beta_1-1}y^{\beta_2}t^{\beta_3}+2\beta_3x^{\beta_1}y^{\beta_2+1}t^{\beta_3-1}=\beta_1 x^{-1}\xi^{\beta}+2\beta_3 yt^{-1}\xi^{\beta},\\
			&Y\xi^{\beta}=\beta_2 x^{\beta_1}y^{\beta_2-1}t^{\beta_3}-2\beta_3x^{\beta_1+1}y^{\beta_2}t^{\beta_3-1}=\beta_2 y^{-1}\xi^{\beta}-2\beta_3 xt^{-1}\xi^{\beta},
		\end{align*}
		which immediately yields the desired equality.
	\end{proof}
	We are now ready to find the correspondent expression of $\mathcal{I}_{u}^{\mathbb{H}^1}$ to \eqref{rewriting-J_u-final}. Precisely, we get its following rewriting. 
	\begin{prop}\label{prop-rewriting-J_u^H^1-final}
		Let $\mathcal{I}_{u}^{\mathbb{H}^1}$ be as in \eqref{functional-harm-funct-H^1}. Then, we have 
		\begin{equation}\label{rewriting-J_u^H^1-final}
			\mathcal{I}_{u}^{\mathbb{H}^1}(r)=\sum_{k=1}^{\infty}r^{2(k-1)}a_{k}^{\mathbb{H}^1}+\sum_{\stackrel{h,k=1}{h\neq k}}^{\infty}r^{h+k-2}a_{h,k}^{\mathbb{H}^1},
		\end{equation}
		where 
		\begin{align}\label{defin-a_k-a_h,k-funct-J_u^H^1}
			&a_{k}^{\mathbb{H}^1}\coloneqq \frac{1}{2k}\int_{\partial B^{\mathbb{H}^1}_1(0)}\frac{Q_k(\sigma)}{\sqrt{\sigma_x^2+\sigma_y^2}}\hspace{0.05cm}d\sigma_{\mathbb{H}^1},\notag\\
			&a_{h,k}^{\mathbb{H}^1}\coloneqq\frac{1}{h+k}\int_{\partial B^{\mathbb{H}^1}_1(0)}\frac{T_{h,k}(\sigma)}{\sqrt{\sigma_x^2+\sigma_y^2}}
			\hspace{0.05cm}d\sigma_{\mathbb{H}^1},\notag\\
		\end{align}
		with
		\begin{equation*}
			\begin{split}
				&Q_k(\sigma)\coloneqq 
				\sum_{\beta,\hspace{0.1cm}|\beta|+\beta_3=k}b_{\beta}^2\sigma^{2\beta}q_k(\beta,\sigma)+\sum_{\stackrel{\beta,\hspace{0.025cm}\gamma,\hspace{0.1cm}|\beta|+\beta_3=k}{|\gamma|+\gamma_3=k,\hspace{0.1cm}\beta\neq \gamma}}b_{\beta}b_{\gamma}\sigma^{\beta+\gamma}q_k(\beta,\gamma,\sigma)\\
				\\
				&q_k(\beta,\sigma)\coloneqq \beta_1^2\sigma_x^{-2}+\beta_2^2 \sigma_y^{-2}+4\beta_3^2\sigma_t^{-2} (\sigma_x^2+\sigma_y^2)+4\beta_3\sigma_t^{-1}(\beta_1\sigma_x^{-1} \sigma_y\\
				&-\beta_2\sigma_x\sigma_y^{-1}),\\
			\end{split}
		\end{equation*}	\begin{equation}\label{defin-polyn-square-norm-grad-homog-polyn-H^1-spher-coord}
			\begin{split}
				\\
				&q_k(\beta,\gamma,\sigma)\coloneqq \beta_1 \gamma_1 \sigma_x^{-2}+\beta_2\gamma_2 \sigma_y^{-2}+4\beta_3 \gamma_3\sigma_t^{-2}(\sigma_x^2+\sigma_y^2)\\ &+2\sigma_t^{-1}((\beta_1\gamma_3+\beta_3\gamma_1)\sigma_x^{-1} \sigma_y-(\beta_2\gamma_3+\beta_3\gamma_2)\sigma_x\sigma_y^{-1}),
			\end{split}
		\end{equation}
		and
		\begin{equation}\label{defin-polyn-scal-prod-grad-homog-polyn-H^1-spher-coord}
			\begin{split}
				&T_{h,k}(\sigma)\coloneqq \sum_{\stackrel{\beta,\hspace{0.025cm}\gamma,\hspace{0.1cm}|\beta|+\beta_3=h,}{|\gamma|+\gamma_3=k}}b_{\beta}b_{\gamma}\sigma^{\beta+\gamma}q_{h,k}(\beta,\gamma,\sigma),\\
				\\
				&q_{h,k}(\beta,\gamma,\sigma)\coloneqq \beta_1 \gamma_1 \sigma_x^{-2}+\beta_2\gamma_2 \sigma_y^{-2}+4\beta_3 \gamma_3\sigma_t^{-2}(\sigma_x^2+\sigma_y^2)\\ &+2\sigma_t^{-1}((\beta_1\gamma_3+\beta_3\gamma_1)\sigma_x^{-1} \sigma_y-(\beta_2\gamma_3+\beta_3\gamma_2)\sigma_x\sigma_y^{-1}),
		\end{split}
		\end{equation}
	\end{prop}
	\begin{proof}
		The proof follows by a rewriting of \eqref{rewriting-J_u^H^1-2}.\\ About the first term in \eqref{rewriting-J_u^H^1-2},
		\[\frac{1}{r^2}\int^r_0s\bigg(\sum_{k=1}^{\infty}\int_{\partial B^{\mathbb{H}^1}_1(0)}\frac{|(\nabla_{\mathbb{H}^1} P_{k})(\delta_s(\sigma))|^2}{\sqrt{\sigma_x^2+\sigma_y^2}}\hspace{0.05cm}d\sigma_{\mathbb{H}^1}\bigg)ds,\]
		using Lemma \ref{lemma-grad-homog-polyn-H^1} and \eqref{homogeneity-xi^beta}, we have
		\begin{align*}
			&|(\nabla_{\mathbb{H}^1} P_{k})(\delta_s(\sigma))|^2=\bigg|\sum_{\beta,\hspace{0.1cm}|\beta|+\beta_3=k}b_{\beta}s^k\sigma^{\beta}(\beta_1 s^{-1}\sigma_x^{-1}+2\beta_3 s^{-1}\sigma_y\sigma_t^{-1},\beta_2 s^{-1}\sigma_y^{-1}\\
			&-2\beta_3 s^{-1}\sigma_x\sigma_t^{-1})\bigg|^2\\
			&=\bigg(\sum_{\beta,\hspace{0.1cm}|\beta|+\beta_3=k}b_{\beta}s^{k-1}\sigma^{\beta}(\beta_1\sigma_x^{-1}+2\beta_3 \sigma_y\sigma_t^{-1})\bigg)^2\\
			&+\bigg(\sum_{\beta,\hspace{0.1cm}|\beta|+\beta_3=k}b_{\beta}s^{k-1}\sigma^{\beta}(\beta_2\sigma_y^{-1}-2\beta_3 \sigma_x\sigma_t^{-1})\bigg)^2\\
			&=\sum_{\beta,\hspace{0.1cm}|\beta|+\beta_3=k}(b_{\beta}s^{k-1}\sigma^{\beta}(\beta_1 \sigma_x^{-1}+2\beta_3\sigma_y\sigma_t^{-1}))^2+2\sum_{\stackrel{\beta,\hspace{0.025cm}\gamma,\hspace{0.1cm}|\beta|+\beta_3=k}{|\gamma|+\gamma_3=k,\hspace{0.1cm}\beta\neq \gamma}}b_{\beta}s^{k-1}\sigma^{\beta}(\beta_1\\
			&\sigma_x^{-1}+2\beta_3 \sigma_y\sigma_t^{-1})b_{\gamma}s^{k-1}\sigma^{\gamma}(\gamma_1 \sigma_x^{-1}+2\gamma_3\sigma_y\sigma_t^{-1})+\sum_{\beta,\hspace{0.1cm}|\beta|+\beta_3=k}(b_{\beta}s^{k-1}\sigma^{\beta}(\beta_2 \sigma_y^{-1}\\
			&-2\beta_3\sigma_x\sigma_t^{-1}))^2+2\sum_{\stackrel{\beta,\hspace{0.025cm}\gamma,\hspace{0.1cm}|\beta|+\beta_3=k}{|\gamma|+\gamma_3=k,\hspace{0.1cm}\beta\neq \gamma}}b_{\beta}s^{k-1}\sigma^{\beta}(\beta_2 \sigma_y^{-1}-2\beta_3 \sigma_x\sigma_t^{-1})b_{\gamma}s^{k-1}\sigma^{\gamma}(\gamma_2 \\
			&\sigma_y^{-1}-2\gamma_3\sigma_x\sigma_t^{-1})\\
			&=s^{2(k-1)}\sum_{\beta,\hspace{0.1cm}|\beta|+\beta_3=k}b_{\beta}^2\sigma^{2\beta}(\beta_1^2\sigma_x^{-2}+4\beta_3^2 \sigma_y^2\sigma_t^{-2}+4\beta_1\beta_3\sigma_x^{-1} \sigma_y\sigma_t^{-1})
	\end{align*}
	\begin{align*}
		    &+2s^{2(k-1)}\sum_{\stackrel{\beta,\hspace{0.025cm}\gamma,\hspace{0.1cm}|\beta|+\beta_3=k}{|\gamma|+\gamma_3=k,\hspace{0.1cm}\beta\neq \gamma}}b_{\beta}b_{\gamma}\sigma^{\beta+\gamma}(\beta_1 \gamma_1 \sigma_x^{-2}+2\beta_3\gamma_1\sigma_x^{-1} \sigma_y\sigma_t^{-1}+2\beta_1\gamma_3\sigma_x^{-1}\sigma_y\\
			&\sigma_t^{-1}+4\beta_3 \gamma_3 \sigma_y^2\sigma_t^{-2})\\
			&+s^{2(k-1)}\sum_{\beta,\hspace{0.1cm}|\beta|+\beta_3=k}b_{\beta}^2\sigma^{2\beta}(\beta_2^2\sigma_y^{-2}+4\beta_3^2 \sigma_x^2\sigma_t^{-2}-4\beta_2\beta_3\sigma_x\sigma_y^{-1}\sigma_t^{-1})\\
			&+2s^{2(k-1)}\sum_{\stackrel{\beta,\hspace{0.025cm}\gamma,\hspace{0.1cm}|\beta|+\beta_3=k}{|\gamma|+\gamma_3=k,\hspace{0.1cm}\beta\neq \gamma}}b_{\beta}b_{\gamma}\sigma^{\beta+\gamma}(\beta_2 \gamma_2 \sigma_y^{-2}-2\beta_3\gamma_2\sigma_x\sigma_y^{-1}\sigma_t^{-1}-2\beta_2\gamma_3\sigma_x\sigma_y^{-1}\\
			&\sigma_t^{-1}+4\beta_3 \gamma_3 \sigma_x^2\sigma_t^{-2})\\
		\end{align*}
		which yields
		\begin{equation}\label{square-norm-grad-homog-polyn-H^1-spher-coord}
			|(\nabla_{\mathbb{H}^1} P_{k})(\delta_s(\sigma))|^2=s^{2(k-1)}Q_k(\sigma),
		\end{equation}
		where we denote $Q_k(\sigma)$ by
		\begin{align*}
			&Q_k(\sigma)\coloneqq 
			\sum_{\beta,\hspace{0.1cm}|\beta|+\beta_3=k}b_{\beta}^2\sigma^{2\beta}q_k(\beta,\sigma)+\sum_{\stackrel{\beta,\hspace{0.025cm}\gamma,\hspace{0.1cm}|\beta|+\beta_3=k}{|\gamma|+\gamma_3=k,\hspace{0.1cm}\beta\neq \gamma}}b_{\beta}b_{\gamma}\sigma^{\beta+\gamma}q_k(\beta,\gamma,\sigma),\\
			\\
			&q_k(\beta,\sigma)\coloneqq \beta_1^2\sigma_x^{-2}+\beta_2^2 \sigma_y^{-2}+4\beta_3^2\sigma_t^{-2} (\sigma_x^2+\sigma_y^2)+4\beta_3\sigma_t^{-1}(\beta_1\sigma_x^{-1} \sigma_y\\
			&-\beta_2\sigma_x\sigma_y^{-1})\\
			\\
			&q_k(\beta,\gamma,\sigma)\coloneqq \beta_1 \gamma_1 \sigma_x^{-2}+\beta_2\gamma_2 \sigma_y^{-2}+4\beta_3 \gamma_3\sigma_t^{-2}(\sigma_x^2+\sigma_y^2)\\ &+2\sigma_t^{-1}((\beta_1\gamma_3+\beta_3\gamma_1)\sigma_x^{-1} \sigma_y-(\beta_2\gamma_3+\beta_3\gamma_2)\sigma_x\sigma_y^{-1}).
		\end{align*}
		Concerning the second term in \eqref{rewriting-J_u^H^1-2},
		\[\frac{1}{r^2}\int^r_0s\bigg(\sum_{\stackrel{h,k=1}{h\neq k}}^{\infty}\int_{\partial B^{\mathbb{H}^1}_1(0)}\frac{\langle(\nabla_{\mathbb{H}^1} P_{h})(\delta_s(\sigma)),(\nabla_{\mathbb{H}^1} P_{k})(\delta_s(\sigma))\rangle}{\sqrt{\sigma_x^2+\sigma_y^2}}\hspace{0.05cm}d\sigma_{\mathbb{H}^1}\bigg)ds,\]
		 we first notice that a priori we do not know if the gradients of two $\mathbb{H}^1$-harmonic homogeneous polynomials of different degrees are orthogonal on the Koranyi unit sphere $\partial B^{\mathbb{H}^1}_1(0),$ in the sense of satisfying
		\[\int_{\partial B^{\mathbb{H}^1}_1(0)}\frac{\langle(\nabla_{\mathbb{H}^1} P_{h})(\delta_s(\sigma)),(\nabla_{\mathbb{H}^1} P_{k})(\delta_s(\sigma))\rangle}{\sqrt{\sigma_x^2+\sigma_y^2}}\hspace{0.05cm}d\sigma_{\mathbb{H}^1}=0,\quad h\neq k.\]
		Thus, we go back to the proof of the correspondent result in the Euclidean case, Lemma \ref{lemma-orthogonality-homog-harm-polyn}. We recall that Lemma \ref{lemma-orthogonality-homog-harm-polyn} is a direct consequence of the harmonicity property of the polynomials and the divergence theorem, where in particular we exploit the fact that the outward normal to $\partial B_1$ is exactly $x.$ In the case of $\mathbb{H}^1,$ we can use the same ingredients, but with the difference that the outward normal to $\partial B^{\mathbb{H}^1}_1(0)$ is 
		\begin{equation}\label{outward-normal-Koranyi-unit-sphere-H^1}
			\nu_{\mathbb{H}^1}(\sigma)\coloneqq \frac{(\nabla_{\mathbb{H}^1}|\xi|_{\mathbb{H}^1})(\sigma)}{|\nabla_{\mathbb{H}^1}|\xi|_{\mathbb{H}^1}|(\sigma)},
		\end{equation}
		if $\sigma$ is not a characteristic point for $\partial B^{\mathbb{H}^1}_1(0),$ see Section \ref{Heisenberg_setting}. This difference could yield in turn that the correspondent result to \eqref{condition-orthogon-homog-harm-polyn} in $\mathbb{H}^1$ is not true. To investigate this, we compute $\langle(\nabla_{\mathbb{H}^1}P_k)(\sigma), \nu_{\mathbb{H}^1}(\sigma)\rangle.$\\
		First of all, we write the explicit expression of $\nu_{\mathbb{H}^1}(\sigma),$ that is, according to 
		\eqref{norm-grad-heisen-koranyi-norm-H^1}, \eqref{outward-normal-Koranyi-unit-sphere-H^1} and the fact that $|\sigma|_{\mathbb{H}^1}=1,$ with $\sigma\in \partial B^{\mathbb{H}^1}_1(0),$ explicitly calculating
		\[\nu_{\mathbb{H}^1}(\sigma)=\frac{((\sigma_x^2+\sigma_y^2)\sigma_x+\sigma_y \sigma_t,(\sigma_x^2+\sigma_y^2)\sigma_y-\sigma_x \sigma_t)}{\sqrt{\sigma_x^2+\sigma_y^2}}.\]
		Hence we obtain, again by virtue of Lemma \ref{lemma-grad-homog-polyn-H^1},
		\begin{align*}
			&\langle(\nabla_{\mathbb{H}^1}P_k)(\sigma), \nu_{\mathbb{H}^1}(\sigma)\rangle=\frac{1}{\sqrt{\sigma_x^2+\sigma_y^2}}\sum_{\beta,\hspace{0.1cm}|\beta|+\beta_3=k}b_{\beta}\sigma^{\beta}(\beta_1(\sigma_x^2+\sigma_y^2)+2\beta_3\sigma_y^2\\
			&+\beta_1\sigma_x^{-1}\sigma_y\sigma_t+2\beta_3(\sigma_x^2+\sigma_y^2)\sigma_x\sigma_y\sigma_t^{-1}+\beta_2(\sigma_x^2+\sigma_y^2)+2\beta_3\sigma_x^2-\beta_2\sigma_x\sigma_y^{-1}\sigma_t\\
			&-2\beta_3(\sigma_x^2+\sigma_y^2)\sigma_x\sigma_y\sigma_t^{-1})=\frac{1}{\sqrt{\sigma_x^2+\sigma_y^2}}\sum_{\beta,\hspace{0.1cm}|\beta|+\beta_3=k}b_{\beta}\sigma^{\beta}((\beta_1+\beta_2+2\beta_3)\\
			&(\sigma_x^2+\sigma_y^2)+(\beta_1\sigma_x^{-1}\sigma_y-\beta_2\sigma_x\sigma_y^{-1})\sigma_t),
		\end{align*}
		which gives
		\begin{equation}\label{scal-prod-grad-homog-polyn-outward-normal-Koranyi-ball-H^1}
			\begin{split}
				&\langle(\nabla_{\mathbb{H}^1}P_k)(\sigma), \nu_{\mathbb{H}^1}(\sigma)\rangle=\frac{1}{\sqrt{\sigma_x^2+\sigma_y^2}}
				\sum_{\beta,\hspace{0.1cm}|\beta|+\beta_3=k}b_{\beta}\sigma^{\beta}(k(\sigma_x^2+\sigma_y^2)+(\beta_1\sigma_x^{-1}\sigma_y\\
				&-\beta_2\sigma_x\sigma_y^{-1})\sigma_t).
			\end{split}
		\end{equation}
		We then note that, differently from Lemma \ref{lemma-orthogonality-homog-harm-polyn}, $\langle(\nabla_{\mathbb{H}^1}P_k)(\sigma), \nu_{\mathbb{H}^1}(\sigma)\rangle$ is not $kP_k(\sigma),$ so we can not expect to achieve the same conclusion \eqref{rewriting-J_u-final} in the case of $\mathbb{H}^1.$ Specifically, if we repeat the same considerations done in the proof of Lemma \ref{lemma-orthogonality-homog-harm-polyn}, we reach the equality
		\begin{equation}\label{condition-orthogon-harm-homog-polyn-H^1}
			0=\int_{\partial B^{\mathbb{H}^1}_1(0)}(P_h(\sigma)\langle(\nabla_{\mathbb{H}^1}P_k)(\sigma), \nu_{\mathbb{H}^1}(\sigma)\rangle-P_k(\sigma)\langle(\nabla_{\mathbb{H}^1}P_h)(\sigma), \nu_{\mathbb{H}^1}(\sigma)\rangle)\hspace{0.05cm}d\sigma_{\mathbb{H}^1},
		\end{equation}
		whereas in \eqref{rewriting-J_u^H^1-2} we have
		\begin{equation}\label{integr-term-scalar-prod-gradients-J_u-H^1}
			\int_{\partial B^{\mathbb{H}^1}_1(0)}\frac{\langle(\nabla_{\mathbb{H}^1} P_{h})(\delta_s(\sigma)),(\nabla_{\mathbb{H}^1} P_{k})(\delta_s(\sigma))\rangle}{\sqrt{\sigma_x^2+\sigma_y^2}}\hspace{0.05cm}d\sigma_{\mathbb{H}^1}.
		\end{equation}
		Consequently, we can try to compare
		\begin{equation}\label{argument-integr-condition-orthogon-harm-homog-polyn-H^1-before-comput}
			P_h(\sigma)\langle\nabla_{\mathbb{H}^1}P_k(\sigma), \nu_{\mathbb{H}^1}(\sigma)\rangle-P_k(\sigma)\langle\nabla_{\mathbb{H}^1}P_h(\sigma), \nu_{\mathbb{H}^1}(\sigma)\rangle
		\end{equation}
		and
		\begin{equation}\label{scalar-prod-gradients-homog-polyn-H^1-weighted-spher-coord-before-comput}
			\frac{\langle(\nabla_{\mathbb{H}^1} P_{h})(\delta_s(\sigma)),(\nabla_{\mathbb{H}^1} P_{k})(\delta_s(\sigma))\rangle}{\sqrt{\sigma_x^2+\sigma_y^2}},
		\end{equation}
		with $h\neq k,$ to understand if the term \eqref{integr-term-scalar-prod-gradients-J_u-H^1} can be zero. Let us compute explicitly again.\\
		First, about \eqref{argument-integr-condition-orthogon-harm-homog-polyn-H^1-before-comput}, we recall \eqref{express-homog-polyn-H^1} and \eqref{scal-prod-grad-homog-polyn-outward-normal-Koranyi-ball-H^1} and we get
		\begin{equation}\label{argument-integr-condition-orthogon-harm-homog-polyn-H^1-computed}
			\begin{split}
				&P_h(\sigma)\langle\nabla_{\mathbb{H}^1}P_k(\sigma), \nu_{\mathbb{H}^1}(\sigma)\rangle-P_k(\sigma)\langle\nabla_{\mathbb{H}^1}P_h(\sigma), \nu_{\mathbb{H}^1}(\sigma)\rangle\\
				&=\frac{1}{\sqrt{\sigma_x^2+\sigma_y^2}}\bigg(\sum_{\stackrel{\beta,\hspace{0.025cm}\gamma,\hspace{0.1cm}|\beta|+\beta_3=h,}{|\gamma|+\gamma_3=k}}b_{\beta}b_{\gamma}\sigma^{\beta+\gamma}(k(\sigma_x^2+\sigma_y^2)+(\gamma_1\sigma_x^{-1}\sigma_y-\gamma_2\sigma_x\sigma_y^{-1})\sigma_t)\\
				&-\sum_{\stackrel{\beta,\hspace{0.025cm}\gamma,\hspace{0.1cm}|\beta|+\beta_3=h,}{|\gamma|+\gamma_3=k}}b_{\beta}b_{\gamma}\sigma^{\beta+\gamma}(h(\sigma_x^2+\sigma_y^2)+(\beta_1\sigma_x^{-1}\sigma_y-\beta_2\sigma_x\sigma_y^{-1})\sigma_t)\bigg)\\
				&=\frac{1}{\sqrt{\sigma_x^2+\sigma_y^2}}\sum_{\stackrel{\beta,\hspace{0.025cm}\gamma,\hspace{0.1cm}|\beta|+\beta_3=h,}{|\gamma|+\gamma_3=k}}b_{\beta}b_{\gamma}\sigma^{\beta+\gamma}((k-h)(\sigma_x^2+\sigma_y^2)+((\gamma_1-\beta_1)\sigma_x^{-1}\sigma_y\\
				&-(\gamma_2-\beta_2)\sigma_x\sigma_y^{-1})\sigma_t).		
			\end{split}
		\end{equation}
		Concerning \eqref{scalar-prod-gradients-homog-polyn-H^1-weighted-spher-coord-before-comput}, we obtain, repeating some of the steps exploited to have \eqref{square-norm-grad-homog-polyn-H^1-spher-coord},
		\begin{align*}
			&\langle(\nabla_{\mathbb{H}^1} P_{h})(\delta_s(\sigma)),(\nabla_{\mathbb{H}^1} P_{k})(\delta_s(\sigma))\rangle=\langle\sum_{\beta,\hspace{0.1cm}|\beta|+\beta_3=h}b_{\beta}s^{h-1}\sigma^{\beta}(\beta_1\sigma_x^{-1}+2\beta_3 \sigma_y\\
			&\sigma_t^{-1},\beta_2 \sigma_y^{-1}-2\beta_3 \sigma_x\sigma_t^{-1}),\sum_{\gamma,\hspace{0.1cm}|\gamma|+\gamma_3=k}b_{\gamma}s^{k-1}\sigma^{\gamma}(\gamma_1 \sigma_x^{-1}+2\gamma_3 \sigma_y\sigma_t^{-1},\gamma_2 \sigma_y^{-1}\\
			&-2\gamma_3 \sigma_x\sigma_t^{-1})\rangle
		\end{align*}
		\begin{align*}
			&=s^{h+k-2}\sum_{\stackrel{\beta,\hspace{0.025cm}\gamma,\hspace{0.1cm}|\beta|+\beta_3=h,}{|\gamma|+\gamma_3=k}}b_{\beta}b_{\gamma}\sigma^{\beta+\gamma}(\beta_1\gamma_1 \sigma_x^{-2}+2\beta_3\gamma_1\sigma_x^{-1}\sigma_y\sigma_t^{-1}+2\beta_1\gamma_3\sigma_x^{-1}\sigma_y\\
			&\sigma_t^{-1}+4\beta_3\gamma_3 \sigma_y^2\sigma_t^{-2}+\beta_2 \gamma_2 \sigma_y^{-2}-2\beta_3\gamma_2 \sigma_x\sigma_y^{-1}\sigma_t^{-1}-2\beta_2\gamma_3\sigma_x\sigma_y^{-1}\sigma_t^{-1}+4\beta_3\\
			&\gamma_3 \sigma_x^2\sigma_t^{-2})=s^{h+k-2}T_{h,k}(\sigma),\\   
		\end{align*}
		which implies
		\begin{equation}\label{scalar-prod-gradients-homog-polyn-H^1-weighted-spher-coord-computed}
			\begin{split}
				&\frac{\langle(\nabla_{\mathbb{H}^1} P_{h})(\delta_s(\sigma)),(\nabla_{\mathbb{H}^1} P_{k})(\delta_s(\sigma))\rangle}{\sqrt{\sigma_x^2+\sigma_y^2}}=\frac{s^{h+k-2}T_{h,k}(\sigma)}{\sqrt{\sigma_x^2+\sigma_y^2}},
			\end{split}
		\end{equation}
		with $T_{h,k}(\sigma)$ defined as 
		\begin{equation*}
			\begin{split}
			&T_{h,k}(\sigma)\coloneqq \sum_{\stackrel{\beta,\hspace{0.025cm}\gamma,\hspace{0.1cm}|\beta|+\beta_3=h,}{|\gamma|+\gamma_3=k}}b_{\beta}b_{\gamma}\sigma^{\beta+\gamma}q_{h,k}(\beta,\gamma,\sigma),\\
			\\
			&q_{h,k}(\beta,\gamma,\sigma)\coloneqq\beta_1\gamma_1 \sigma_x^{-2}+\beta_2 \gamma_2 \sigma_y^{-2}+2(\beta_1\gamma_3+\beta_3\gamma_1)\sigma_x^{-1}\sigma_y\sigma_t^{-1}\\
			&-2(\beta_2\gamma_3+\beta_3\gamma_2)\sigma_x\sigma_y^{-1}\sigma_t^{-1}+4\beta_3\gamma_3(\sigma_x^2+\sigma_y^2)\sigma_t^{-2}.  
			\end{split}
		\end{equation*}
		At this point, if we compare \eqref{argument-integr-condition-orthogon-harm-homog-polyn-H^1-computed} and \eqref{scalar-prod-gradients-homog-polyn-H^1-weighted-spher-coord-computed}, we see that the two terms are not the same. Therefore, we can not expect to have \eqref{integr-term-scalar-prod-gradients-J_u-H^1} equal to $0$ in \eqref{rewriting-J_u^H^1-2} and the correspondent expression to \eqref{rewriting-J_u-final} in $\mathbb{H}^1.$\\
		Let us look then at what we achieve substituting \eqref{square-norm-grad-homog-polyn-H^1-spher-coord} and \eqref{scalar-prod-gradients-homog-polyn-H^1-weighted-spher-coord-computed} in \eqref{rewriting-J_u^H^1-2}. We find
		\begin{align*}
			&\mathcal{I}_{u}^{\mathbb{H}^1}(r)=\frac{1}{r^2}\int^r_0s\bigg(\sum_{k=1}^{\infty}\int_{\partial B^{\mathbb{H}^1}_1(0)}\frac{s^{2(k-1)}Q_k(\sigma)}{\sqrt{\sigma_x^2+\sigma_y^2}}\hspace{0.05cm}d\sigma_{\mathbb{H}^1}\bigg)ds\\	
			&+\frac{1}{r^2}\int^r_0s\bigg(\sum_{\stackrel{h,k=1}{h\neq k}}^{\infty}\int_{\partial B^{\mathbb{H}^1}_1(0)}\frac{s^{h+k-2}T_{h,k}(\sigma)}{\sqrt{\sigma_x^2+\sigma_y^2}}
			\hspace{0.05cm}d\sigma_{\mathbb{H}^1}\bigg)ds\\
			&=\frac{1}{r^2}\bigg(\sum_{k=1}^{\infty}\int_{\partial B^{\mathbb{H}^1}_1(0)}\frac{Q_k(\sigma)}{\sqrt{\sigma_x^2+\sigma_y^2}}\bigg(\int^r_0s^{2k-1}ds\bigg)\hspace{0.05cm}d\sigma_{\mathbb{H}^1}\\
			&+\sum_{\stackrel{h,k=1}{h\neq k}}^{\infty}\int_{\partial B^{\mathbb{H}^1}_1(0)}\frac{T_{h,k}(\sigma)}{\sqrt{\sigma_x^2+\sigma_y^2}}\bigg(\int^r_0s^{h+k-1}ds\bigg)\hspace{0.05cm}d\sigma_{\mathbb{H}^1}\bigg)\\
			&=\frac{1}{r^2}\bigg(\sum_{k=1}^{\infty}\frac{r^{2k}}{2k}\int_{\partial B^{\mathbb{H}^1}_1(0)}\frac{Q_k(\sigma)}{\sqrt{\sigma_x^2+\sigma_y^2}}\hspace{0.05cm}d\sigma_{\mathbb{H}^1}+\sum_{\stackrel{h,k=1}{h\neq k}}^{\infty}\frac{r^{h+k}}{h+k}\int_{\partial B^{\mathbb{H}^1}_1(0)}\frac{T_{h,k}(\sigma)}{\sqrt{\sigma_x^2+\sigma_y^2}}\hspace{0.05cm}d\sigma_{\mathbb{H}^1}\bigg),
		\end{align*}
		which gives
		\begin{equation*}
			\mathcal{I}_{u}^{\mathbb{H}^1}(r)=\sum_{k=1}^{\infty}r^{2(k-1)}a_{k}^{\mathbb{H}^1}+\sum_{\stackrel{h,k=1}{h\neq k}}^{\infty}r^{h+k-2}a_{h,k}^{\mathbb{H}^1},\\
		\end{equation*}
		where 
		\begin{align*}
			&a_{k}^{\mathbb{H}^1}\coloneqq \frac{1}{2k}\int_{\partial B^{\mathbb{H}^1}_1(0)}\frac{Q_k(\sigma)}{\sqrt{\sigma_x^2+\sigma_y^2}}\hspace{0.05cm}d\sigma_{\mathbb{H}^1},\\
			&a_{h,k}^{\mathbb{H}^1}\coloneqq\frac{1}{h+k}\int_{\partial B^{\mathbb{H}^1}_1(0)}\frac{T_{h,k}(\sigma)}{\sqrt{\sigma_x^2+\sigma_y^2}}
			\hspace{0.05cm}d\sigma_{\mathbb{H}^1}.
		\end{align*}
	\end{proof}
	So, the different structure in $\mathbb{H}^1$ with respect to the Euclidean one yields a different expression for the functional $\mathcal{I}_{u}^{\mathbb{H}^1}$ compared to $\mathcal{I}_{u}.$\\
	However, we can deduce the correspondent statement to Proposition \ref{prop-limit-r-tend-to-0-monotonic-J_u}.
	\begin{prop}\label{prop-limit-r-tend-to-0-monotonic-J_u^H^1}
		Let $\mathcal{I}_{u}^{\mathbb{H}^1}$ be as in \eqref{functional-harm-funct-H^1}. Then  $\lim\limits_{r\rightarrow 0}\mathcal{I}_{u}^{\mathbb{H}^1}(r)=a_1^{\mathbb{H}^1},$ with $a_1^{\mathbb{H}^1}$ defined as in \eqref{defin-a_k-a_h,k-funct-J_u^H^1}, and the monotonicity behavior of $\mathcal{I}_{u}^{\mathbb{H}^1}$ around $r=0$ depends on the sign of $a_{2,1}^{\mathbb{H}^1},$ if this term is different from $0.$
	\end{prop}
	\begin{proof}
		The proof directly follows by Proposition \ref{prop-rewriting-J_u^H^1-final}. First, letting $r$ tend to $0$ in \eqref{rewriting-J_u^H^1-final}, we get
		\[
		\lim\limits_{r\rightarrow 0}\mathcal{I}_{u}^{\mathbb{H}^1}(r)=\lim\limits_{r\rightarrow 0}\bigg(\sum_{k=1}^{\infty}r^{2(k-1)}a_{k}^{\mathbb{H}^1}+\sum_{\stackrel{h,k=1}{h\neq k}}^{\infty}r^{h+k-2}a_{h,k}^{\mathbb{H}^1}\bigg)=a_1^{\mathbb{H}^1}.\\
		\]
		Next, for the condition on the monotonicity behavior of $\mathcal{I}_{u}^{\mathbb{H}^1},$ from \eqref{rewriting-J_u^H^1-final} we have
		\[(\mathcal{I}_{u}^{\mathbb{H}^1})'(r)=\sum_{k=2}^{\infty}2(k-1)r^{2k-3}a_{k}^{\mathbb{H}^1}+\sum_{\stackrel{h,k=1}{h\neq k}}^{\infty}(h+k-2)r^{h+k-3}a_{h,k}^{\mathbb{H}^1}.\\
		\]
		This expression then entails
		\[\lim\limits_{r\rightarrow 0}(\mathcal{I}_{u}^{\mathbb{H}^1})'(r)=2a_{2,1}^{\mathbb{H}^1},\]
		using that $a_{2,1}^{\mathbb{H}^1}=a_{1,2}^{\mathbb{H}^1}$ by \eqref{defin-a_k-a_h,k-funct-J_u^H^1}. Hence, the monotonicity behavior of $\mathcal{I}_{u}^{\mathbb{H}^1}$ around $r=0$ depends on the sign of
		$a_{2,1}^{\mathbb{H}^1},$
		if $a_{2,1}^{\mathbb{H}^1}$ is different from $0.$ 
	\end{proof}
	Let us discuss now, as in the Euclidean case, the condition \[\lim\limits_{r\rightarrow 0}\mathcal{I}_{u}^{\mathbb{H}^1}(r)\allowbreak=a_1^{\mathbb{H}^1}\] more.
	\begin{prop}
		Let $a_{1}^{\mathbb{H}^1}$ be as in \eqref{defin-a_k-a_h,k-funct-J_u^H^1}. It holds
		\[a_{1}^{\mathbb{H}^1}=\frac{|\nabla_{\mathbb{H}^1} u(0)|^2}{2}\int_{\partial B^{\mathbb{H}^1}_1(0)}\frac{1}{\sqrt{\sigma_x^2+\sigma_y^2}}\hspace{0.05cm}d\sigma_{\mathbb{H}^1}.\]
	\end{prop}
	\begin{proof}
		According to \eqref{defin-a_k-a_h,k-funct-J_u^H^1}, we have
		\[a_{1}^{\mathbb{H}^1}=\frac{1}{2}\int_{\partial B^{\mathbb{H}^1}_1(0)}\frac{Q_{1}(\sigma)}{\sqrt{\sigma_x^2+\sigma_y^2}}\hspace{0.05cm}d\sigma_{\mathbb{H}^1}.\]
		In particular, recalling \eqref{defin-polyn-square-norm-grad-homog-polyn-H^1-spher-coord}, $Q_1$ is the square norm of the $\mathbb{H}^1$-gradient of the homogeneous polynomial of degree $1,$ $P_1,$ coming from the Taylor expansion of $u$ at $0.$ Therefore, exploiting Lemma \ref{lemma-grad-homog-polyn-H^1} and repeating the considerations done in the Euclidean case, we obtain the thesis.
	\end{proof}
	\begin{rem}
		The equality above shows us that the limit of $\mathcal{I}_{u}^{\mathbb{H}^1}$ with $r$ tending to $0$ is strictly positive depending on whether the $\mathbb{H}^1$-gradient of $u$ vanishes at $0$ or not, in other words whether $0$ is a characteristic point of $\{u=u(0)\}$ or not. Indeed, we have
		\[\int_{\partial B^{\mathbb{H}^1}_1(0)}\frac{1}{\sqrt{\sigma_x^2+\sigma_y^2}}\hspace{0.05cm}d\sigma_{\mathbb{H}^1}>0,\]
		since it is the integral of a strictly positive function on $\partial B^{\mathbb{H}^1}_1(0).$
	\end{rem}
	Moreover, we want to rewrite $a_{2,1}^{\mathbb{H}^1}$ to deal with a more explicit tool. Specifically, we have the following.
	\begin{prop}
		Let $a_{2,1}^{\mathbb{H}^1}$ be as in  \eqref{defin-a_k-a_h,k-funct-J_u^H^1}. Then, we have
		\begin{equation}\label{rewriting-monoton-condit-J_u^H^1-final}
			\begin{split}
				&a_{2,1}^{\mathbb{H}^1}=\frac{1}{3}\int_{\partial B^{\mathbb{H}^1}_1(0)}\frac{(2b_{(2,0,0)}b_{(1,0,0)}+b_{(0,1,0)}(b_{(1,1,0)}-2b_{(0,0,1)}))\sigma_x}{\sqrt{\sigma_x^2+\sigma_y^2}}\hspace{0.05cm}d\sigma_{\mathbb{H}^1}\\
				&+\frac{1}{3}\int_{\partial B^{\mathbb{H}^1}_1(0)}\frac{(2b_{(0,2,0)}b_{(0,1,0)}+b_{(1,0,0)}(b_{(1,1,0)}+2b_{(0,0,1)}))\sigma_y}{\sqrt{\sigma_x^2+\sigma_y^2}}\hspace{0.05cm}d\sigma_{\mathbb{H}^1}.\\
			\end{split}
		\end{equation}
	\end{prop}
	\begin{proof}
		By virtue of \eqref{defin-polyn-scal-prod-grad-homog-polyn-H^1-spher-coord}, it holds
		\begin{equation}\label{rewriting-monoton-condit-J_u^H^1-1}
			\begin{split}
				&T_{2,1}(\sigma)=  \sum_{\stackrel{\beta,\hspace{0.025cm}\gamma,\hspace{0.1cm}|\beta|+\beta_3=2,}{|\gamma|+\gamma_3=1}}b_{\beta}b_{\gamma}\sigma^{\beta+\gamma}(\beta_1\gamma_1 \sigma_x^{-2}+\beta_2 \gamma_2 \sigma_y^{-2}+2(\beta_1\gamma_3+\beta_3\gamma_1)\sigma_x^{-1}\sigma_y\\
				&\sigma_t^{-1}-2(\beta_2\gamma_3+\beta_3\gamma_2)\sigma_x\sigma_y^{-1}\sigma_t^{-1}+4\beta_3\gamma_3(\sigma_x^2+\sigma_y^2)\sigma_t^{-2}).\\  
			\end{split}
		\end{equation}
		In particular, we remark that we can not have, at the same time, $|\gamma|+\gamma_3=1$ and $\gamma_3\neq 0,$ because if $\gamma_3\neq 0,$ $|\gamma|+\gamma_3\ge 2.$ Therefore, the terms in \eqref{rewriting-monoton-condit-J_u^H^1-1} all satisfy $\gamma_3=0$ and the 
		only possibilities for $\gamma$ are $\gamma=(1,0,0)$ and $\gamma=(0,1,0).$ So, we can simplify \eqref{rewriting-monoton-condit-J_u^H^1-1} and we reach
		\begin{equation}\label{rewriting-monoton-condit-J_u^H^1-2}
			\begin{split}
				&T_{2,1}(\sigma)=  \sum_{\beta,\hspace{0.1cm}|\beta|+\beta_3=2}b_{\beta}b_{(1,0,0)}\sigma^{\beta+(1,0,0)}(\beta_1\sigma_x^{-2}+2\beta_3\sigma_x^{-1}\sigma_y\sigma_t^{-1})\\
				&+\sum_{\beta,\hspace{0.1cm}|\beta|+\beta_3=2}b_{\beta}b_{(0,1,0)}\sigma^{\beta+(0,1,0)}(\beta_2\sigma_y^{-2}-2\beta_3\sigma_x\sigma_y^{-1}\sigma_t^{-1}).\\  
			\end{split}
		\end{equation}
		Now, we first focus on 
		\begin{equation}\label{first-polyn-monoton-condit-J_u^H^1}
			\sum_{\beta,\hspace{0.1cm}|\beta|+\beta_3=2}b_{\beta}b_{(1,0,0)}\sigma^{\beta+(1,0,0)}(\beta_1\sigma_x^{-2}+2\beta_3\sigma_x^{-1}\sigma_y\sigma_t^{-1}).
		\end{equation}
		Since $|\beta|+\beta_3=2,$ it can hold $\beta_1=1$ or $\beta_1=2.$ If $\beta_1=2,$ $\beta$ has to be \linebreak $\beta=(2,0,0),$ while if $\beta_1=1,$ we have $\beta=(1,1,0),$ otherwise $|\beta|+\beta_3\ge 3.$ In parallel, to obtain $\beta_3\neq 0,$ the only possible choice is $\beta=(0,0,1),$ otherwise again $|\beta|+\beta_3\ge 3.$ To recap, in \eqref{first-polyn-monoton-condit-J_u^H^1} we have the terms
		\begin{align*}
			&b_{(2,0,0)}b_{(1,0,0)}\sigma^{(3,0,0)}2\sigma_x^{-2}+b_{(1,1,0)}b_{(1,0,0)}\sigma^{(2,1,0)}\sigma_x^{-2}+2b_{(0,0,1)}b_{(1,0,0)}\sigma^{(1,0,1)}\sigma_x^{-1}\sigma_y\\
			&\sigma_t^{-1}=2b_{(2,0,0)}b_{(1,0,0)}\sigma_x+b_{(1,1,0)}b_{(1,0,0)}\sigma_y+2b_{(0,0,1)}b_{(1,0,0)}\sigma_y,
		\end{align*}
		in other words
		\begin{equation}\label{semplif-first-polyn-monoton-condit-J_u^H^1-1}
			2b_{(2,0,0)}b_{(1,0,0)}\sigma_x+b_{(1,0,0)}(b_{(1,1,0)}+2b_{(0,0,1)})\sigma_y.
		\end{equation}
		Next, let us look at 
		\begin{equation*}
			\sum_{\beta,\hspace{0.1cm}|\beta|+\beta_3=2}b_{\beta}b_{(0,1,0)}\sigma^{\beta+(0,1,0)}(\beta_2\sigma_y^{-2}-2\beta_3\sigma_x\sigma_y^{-1}\sigma_t^{-1}),
		\end{equation*}
		and since the structure is very similar to \eqref{first-polyn-monoton-condit-J_u^H^1}, we can repeat the same considerations with $\beta_2$ instead of $\beta_1$ and we have
		\begin{align*}
			&b_{(0,2,0)}b_{(0,1,0)}\sigma^{(0,3,0)}2\sigma_y^{-2}+b_{(1,1,0)}b_{(0,1,0)}\sigma^{(1,2,0)}\sigma_y^{-2}-2b_{(0,0,1)}b_{(0,1,0)}\sigma^{(0,1,1)}\sigma_x\sigma_y^{-1}\\
			&\sigma_t^{-1}=2b_{(0,2,0)}b_{(0,1,0)}\sigma_y+b_{(1,1,0)}b_{(0,1,0)}\sigma_x-2b_{(0,0,1)}b_{(0,1,0)}\sigma_x,
		\end{align*}
		i.e.
		\begin{equation}\label{semplif-first-polyn-monoton-condit-J_u^H^1-2}
			2b_{(0,2,0)}b_{(0,1,0)}\sigma_y+b_{(0,1,0)}(b_{(1,1,0)}-2b_{(0,0,1)})\sigma_x.
		\end{equation}
		Consequently, according to \eqref{semplif-first-polyn-monoton-condit-J_u^H^1-1} and \eqref{semplif-first-polyn-monoton-condit-J_u^H^1-2}, \eqref{rewriting-monoton-condit-J_u^H^1-2} becomes
		\begin{align*}
			&T_{2,1}(\sigma)=2b_{(2,0,0)}b_{(1,0,0)}\sigma_x+b_{(1,0,0)}(b_{(1,1,0)}+2b_{(0,0,1)})\sigma_y+2b_{(0,2,0)}b_{(0,1,0)}\sigma_y\\
			&+b_{(0,1,0)}(b_{(1,1,0)}-2b_{(0,0,1)})\sigma_x=(2b_{(2,0,0)}b_{(1,0,0)}+b_{(0,1,0)}(b_{(1,1,0)}-2b_{(0,0,1)}))\sigma_x\\
			&(2b_{(0,2,0)}b_{(0,1,0)}+b_{(1,0,0)}(b_{(1,1,0)}+2b_{(0,0,1)}))\sigma_y.
		\end{align*}
		Recalling \eqref{defin-a_k-a_h,k-funct-J_u^H^1}, we then achieve the desired expression of $a_{2,1}^{\mathbb{H}^1}.$
	\end{proof}
	\begin{rem}
		Looking into \eqref{rewriting-monoton-condit-J_u^H^1-final} carefully, we point out that each term is, up to a constant, of the form
		\[\int_{\partial B^{\mathbb{H}^1}_1(0)}\frac{\sigma_i}{\sqrt{\sigma_x^2+\sigma_y^2}}
		\hspace{0.05cm}d\sigma_{\mathbb{H}^1},\]
		with $i=x,$ or $i=y.$ In particular, since $\partial B^{\mathbb{H}^1}_1(0)$ is symmetric with respect to the $i$-th direction, $i=x,y,$ it holds
		\begin{equation}\label{integr-ith-coord-weight-=-0}
			\int_{\partial B^{\mathbb{H}^1}_1(0)}\frac{\sigma_i}{\sqrt{\sigma_x^2+\sigma_y^2}}
			\hspace{0.05cm}d\sigma_{\mathbb{H}^1}=0.
		\end{equation}
		As a consequence, we get $a_{2,1}^{\mathbb{H}^1}=0.$
	\end{rem}
	This fact forces us to revisit the condition for the monotonicity behavior of $\mathcal{I}_{u}^{\mathbb{H}^1}$ around $r=0.$ Indeed, $a_{2,1}^{\mathbb{H}^1}=0$ tells us, from the proof of Proposition \ref{prop-limit-r-tend-to-0-monotonic-J_u^H^1}, that $\lim\limits_{r\rightarrow 0}(\mathcal{I}_{u}^{\mathbb{H}^1})'(r)=0,$ which does not yield a sign of $(\mathcal{I}_{u}^{\mathbb{H}^1})'$ around $r=0.$\\
	Let us then analyze $(\mathcal{I}_{u}^{\mathbb{H}^1})'$ more. To this end, looking at the proof of Proposition \ref{prop-limit-r-tend-to-0-monotonic-J_u^H^1}, we recall that we have
	\begin{align}\label{expression-derivat-J_u^H^1-1}
		(\mathcal{I}_{u}^{\mathbb{H}^1})'(r)=\sum_{k=2}^{\infty}2(k-1)r^{2k-3}a_{k}^{\mathbb{H}^1}+\sum_{\stackrel{h,k=1}{h\neq k}}^{\infty}(h+k-2)r^{h+k-3}a_{h,k}^{\mathbb{H}^1}.
	\end{align}
	Hence, for small radii the terms with $r$ to power $1$ are the ones which establish the sign of $(\mathcal{I}_{u}^{\mathbb{H}^1})'.$ Let us focus  on these terms. Precisely, we obtain the next result.
	\begin{prop}\label{prop-monoton-behav-I_u^H^1-bis}
		Let $a_{2}^{\mathbb{H}^1}$ and $a_{3,1}^{\mathbb{H}^1}$ be as in \eqref{defin-a_k-a_h,k-funct-J_u^H^1}. We have that $a_{2}^{\mathbb{H}^1}$ is nonnegative, whereas $a_{3,1}^{\mathbb{H}^1}$ depends on the function $u$ defining $\mathcal{I}_{u}^{\mathbb{H}^1}.$
	\end{prop} 
	\begin{proof}
		First, we rewrite \eqref{expression-derivat-J_u^H^1-1} to obtain
		\begin{equation}\label{expression-derivat-J_u^H^1-2}
			\begin{split}
				&(\mathcal{I}_{u}^{\mathbb{H}^1})'(r)=2a_{2}^{\mathbb{H}^1}r+\sum_{k=3}^{\infty}2(k-1)r^{2k-3}a_{k}^{\mathbb{H}^1}+2a_{3,1}^{\mathbb{H}^1}r+2a_{1,3}^{\mathbb{H}^1}r\\
				&+\sum_{\stackrel{h,k=2}{h\neq k}}^{\infty}(h+k-2)r^{h+k-3}a_{h,k}^{\mathbb{H}^1}=2(a_{2}^{\mathbb{H}^1}+2a_{3,1}^{\mathbb{H}^1})r +\sum_{k=3}^{\infty}2(k-1)r^{2k-3}a_{k}^{\mathbb{H}^1}\\
				&+\sum_{\stackrel{h,k=2}{h\neq k}}^{\infty}(h+k-2)r^{h+k-3}a_{h,k}^{\mathbb{H}^1},\\
			\end{split}
		\end{equation}
		where we have exploited again that $a_{3,1}^{\mathbb{H}^1}=a_{1,3}^{\mathbb{H}^1}$ by \eqref{defin-a_k-a_h,k-funct-J_u^H^1}. We want to study at this point	$a_{2}^{\mathbb{H}^1}+2a_{3,1}^{\mathbb{H}^1}.$\\
		Looking into \eqref{defin-a_k-a_h,k-funct-J_u^H^1}, we note that $a_{2}^{\mathbb{H}^1}$ involves the polynomial $Q_2,$ see \eqref{defin-polyn-square-norm-grad-homog-polyn-H^1-spher-coord}. This polynomial is made up of monomials which are either positive or with at least one of the variables with odd exponent. Since \eqref{integr-ith-coord-weight-=-0} is true in the same way for monomials with at least one of the variables with odd exponent, just the positive monomials give a contribution in $a_{2}^{\mathbb{H}^1}.$ Consequently, $a_{2}^{\mathbb{H}^1}$ is nonnegative.\\
		In parallel, if we look at the definition of $a_{3,1}^{\mathbb{H}^1},$ again in \eqref{defin-a_k-a_h,k-funct-J_u^H^1}, the situation is similar, but with the difference that we have monomials with all the variables with even exponent where the coefficients are in the form $b_{\beta}b_{\gamma}.$ Thus the sign of such monomials depends on the sign of $b_{\beta}b_{\gamma}.$ In other words, in principle, the sign of $a_{3,1}^{\mathbb{H}^1}$ is related to the specific function $u$ which defines $\mathcal{I}_{u}^{\mathbb{H}^1}.$
	\end{proof}
	\begin{cor}
		Let $u=x-3yt-2x^3.$ Then, we have $a_{2}^{\mathbb{H}^1}+2a_{3,1}^{\mathbb{H}^1}<0.$
	\end{cor}
		\begin{proof}
		We first analyze the explicit expression of $T_{3,1},$ defined in \eqref{defin-polyn-scal-prod-grad-homog-polyn-H^1-spher-coord}. Specifically, we have
	\[
	\begin{split}
		&T_{3,1}(\sigma)= \sum_{\stackrel{\beta,\hspace{0.025cm}\gamma,\hspace{0.1cm}|\beta|+\beta_3=3,}{|\gamma|+\gamma_3=1}}b_{\beta}b_{\gamma}\sigma^{\beta+\gamma}(\beta_1\gamma_1 \sigma_x^{-2}+\beta_2 \gamma_2 \sigma_y^{-2}+2(\beta_1\gamma_3+\beta_3\gamma_1)\sigma_x^{-1}\sigma_y\\
		&\sigma_t^{-1}-2(\beta_2\gamma_3+\beta_3\gamma_2)\sigma_x\sigma_y^{-1}\sigma_t^{-1}+4\beta_3\gamma_3(\sigma_x^2+\sigma_y^2)\sigma_t^{-2}).
	\end{split}
	\]
	Repeating the same argument used to reach \eqref{rewriting-monoton-condit-J_u^H^1-2}, this becomes
	\begin{equation}\label{rewriting-T_3,1-1}
		\begin{split}
			&T_{3,1}(\sigma)=  \sum_{\beta,\hspace{0.1cm}|\beta|+\beta_3=3}b_{\beta}b_{(1,0,0)}\sigma^{\beta+(1,0,0)}(\beta_1\sigma_x^{-2}+2\beta_3\sigma_x^{-1}\sigma_y\sigma_t^{-1})\\
			&+\sum_{\beta,\hspace{0.1cm}|\beta|+\beta_3=3}b_{\beta}b_{(0,1,0)}\sigma^{\beta+(0,1,0)}(\beta_2\sigma_y^{-2}-2\beta_3\sigma_x\sigma_y^{-1}\sigma_t^{-1}).\\  
		\end{split}
	\end{equation}
	Let us focus first on
	\[
	\sum_{\beta,\hspace{0.1cm}|\beta|+\beta_3=3}b_{\beta}b_{(1,0,0)}\sigma^{\beta+(1,0,0)}(\beta_1\sigma_x^{-2}+2\beta_3\sigma_x^{-1}\sigma_y\sigma_t^{-1}),
	\]
	and arguing in a similar way as before for obtaining \eqref{semplif-first-polyn-monoton-condit-J_u^H^1-1}, it holds
	\begin{equation}\label{semplif-first-term-T_3,1}
		\begin{split}
			&\sum_{\beta,\hspace{0.1cm}|\beta|+\beta_3=3}b_{\beta}b_{(1,0,0)}\sigma^{\beta+(1,0,0)}(\beta_1\sigma_x^{-2}+2\beta_3\sigma_x^{-1}\sigma_y\sigma_t^{-1})=2b_{(2,1,0)}b_{(1,0,0)}\sigma_x\sigma_y\\
			&+b_{(1,2,0)}b_{(1,0,0)}\sigma_y^2+3b_{(3,0,0)}b_{(1,0,0)}\sigma_x^2+2b_{(1,0,1)}b_{(1,0,0)}\sigma_x\sigma_y+b_{(1,0,1)}b_{(1,0,0)}\sigma_t\\
			&+2b_{(0,1,1)}b_{(1,0,0)}\sigma_y^2=3b_{(3,0,0)}b_{(1,0,0)}\sigma_x^2+b_{(1,0,0)}(2b_{(0,1,1)}+b_{(1,2,0)})\sigma_y^2\\
			&+2b_{(1,0,0)}(b_{(2,1,0)}+b_{(1,0,1)})\sigma_x\sigma_y+b_{(1,0,1)}b_{(1,0,0)}\sigma_t,
		\end{split}
	\end{equation}
	while for
	\[\sum_{\beta,\hspace{0.1cm}|\beta|+\beta_3=3}b_{\beta}b_{(0,1,0)}\sigma^{\beta+(0,1,0)}(\beta_2\sigma_y^{-2}-2\beta_3\sigma_x\sigma_y^{-1}\sigma_t^{-1}),\]
	we achieve
	\begin{equation}\label{semplif-second-term-T_3,1}
		\begin{split}
			&\sum_{\beta,\hspace{0.1cm}|\beta|+\beta_3=3}b_{\beta}b_{(0,1,0)}\sigma^{\beta+(0,1,0)}(\beta_2\sigma_y^{-2}-2\beta_3\sigma_x\sigma_y^{-1}\sigma_t^{-1})=b_{(2,1,0)}b_{(0,1,0)}\sigma_x^2\\
			&+2b_{(1,2,0)}b_{(0,1,0)}\sigma_x\sigma_y+3b_{(0,3,0)}b_{(0,1,0)}\sigma_y^2-2b_{(0,1,1)}b_{(0,1,0)}\sigma_x\sigma_y+b_{(0,1,1)}b_{(0,1,0)}\sigma_t\\
			&-2b_{(1,0,1)}b_{(0,1,0)}\sigma_x^2=b_{(0,1,0)}(b_{(2,1,0)}-2b_{(1,0,1)})\sigma_x^2+3b_{(0,3,0)}b_{(0,1,0)}\sigma_y^2\\
			&+2b_{(0,1,0)}(b_{(1,2,0)}-b_{(0,1,1)})\sigma_x\sigma_y+b_{(0,1,1)}b_{(0,1,0)}\sigma_t.
		\end{split}
	\end{equation}
	Now, the terms in \eqref{semplif-first-term-T_3,1} and \eqref{semplif-second-term-T_3,1} with $\sigma_x\sigma_y$ and $\sigma_t$ give null contribution in $a_{3,1}^{\mathbb{H}^1},$ for \eqref{integr-ith-coord-weight-=-0} extended to monomials with at least one of the variables with odd exponent. Consequently, in $a_{3,1}^{\mathbb{H}^1}$ it is sufficient to consider 
	\begin{equation}\label{semplif-T_3,1}
		\begin{split}
			&3b_{(3,0,0)}b_{(1,0,0)}\sigma_x^2+b_{(1,0,0)}(2b_{(0,1,1)}+b_{(1,2,0)})\sigma_y^2+b_{(0,1,0)}(b_{(2,1,0)}-2b_{(1,0,1)})\sigma_x^2\\
			&+3b_{(0,3,0)}b_{(0,1,0)}\sigma_y^2=(3b_{(3,0,0)}b_{(1,0,0)}+b_{(2,1,0)}b_{(0,1,0)}-2b_{(1,0,1)}b_{(0,1,0)})\sigma_x^2\\
			&+(3b_{(0,3,0)}b_{(0,1,0)}+b_{(1,2,0)}b_{(1,0,0)}+2b_{(0,1,1)}b_{(1,0,0)})\sigma_y^2.
		\end{split}
	\end{equation}
	We notice that $b_{(1,0,0)},$ $b_{(0,1,0)},$ and $b_{(3,0,0)},$ $b_{(2,1,0)},$ $b_{(1,0,1)},$ $b_{(0,3,0)},$ $b_{(1,2,0)}$ $b_{(0,1,1)}$ are the coefficients of $P_1$ and $P_3$ respectively, where $P_1$ and $P_3$ are the $\mathbb{H}^1$-harmonic homogeneous polynomials of degree $1$ and degree $3$ in the expansion of $u.$\\
	Therefore, to investigate the behavior of $a_{3,1}^{\mathbb{H}^1},$ we need to choose the coefficients of $P_1$ and $P_3$ in such a way that they are both $\mathbb{H}^1$-harmonic. Let us focus on these choices.\\
	Specifically, $b_{(1,0,0)}$ and $b_{(0,1,0)}$ are the coefficients of $x$ and $y$ respectively, which are $\mathbb{H}^1$-harmonic, thus we do not have particular conditions on them. About the coefficients of $P_3,$ we directly compute which relationships they have to satisfy. We get
	\begin{align*}
		&XP_3=3b_{(3,0,0)}x^2+2b_{(2,1,0)}xy+b_{(1,0,1)}t+b_{(1,2,0)}y^2+2b_{(1,0,1)}xy+2b_{(0,1,1)}y^2,\\
		&YP_3=b_{(2,1,0)}x^2+3b_{(0,3,0)}y^2+2b_{(1,2,0)}xy+b_{(0,1,1)}t-2b_{(1,0,1)}x^2-2b_{(0,1,1)}xy,
	\end{align*}
	which gives
	\begin{align*}
		&X^2P_3=6b_{(3,0,0)}x+2b_{(2,1,0)}y+2b_{(1,0,1)}y+2b_{(1,0,1)}y=6b_{(3,0,0)}x+2b_{(2,1,0)}y\\
		&+4b_{(1,0,1)}y,\\
		&Y^2P_3=6b_{(0,3,0)}y+2b_{(1,2,0)}x-2b_{(0,1,1)}x-2b_{(0,1,1)}x=6b_{(0,3,0)}y+2b_{(1,2,0)}x\\
		&-4b_{(0,1,1)}x.
	\end{align*}
	This implies
	\begin{align*}
		&\Delta_{\mathbb{H}^1}P_3=2(3b_{(3,0,0)}+b_{(1,2,0)}-2b_{(0,1,1)})x+2(3b_{(0,3,0)}+b_{(2,1,0)}+2b_{(1,0,1)})y.
	\end{align*}
	So, $P_3$ is $\mathbb{H}^1$-harmonic if the following conditions hold:
	\begin{equation}\label{cond-harm-homog-polyn-degree-3-H^1}
		\begin{cases}
			3b_{(3,0,0)}+b_{(1,2,0)}-2b_{(0,1,1)}=0,\\
			3b_{(0,3,0)}+b_{(2,1,0)}+2b_{(1,0,1)}=0.
		\end{cases}
	\end{equation}
	On the other hand, we note that,  in \eqref{semplif-T_3,1}, $b_{(1,0,0)}$ multiplies the coefficients which appear in the first condition of \eqref{cond-harm-homog-polyn-degree-3-H^1} and $b_{(0,1,0)}$ those in the second one. We choose then $b_{(1,0,0)}=1$ and $b_{(0,1,0)}=0.$ As a consequence, \eqref{semplif-T_3,1} is reduced to
	\[
	3b_{(3,0,0)}\sigma_x^2+(b_{(1,2,0)}+2b_{(0,1,1)})\sigma_y^2,	
	\]
	but we know that \eqref{cond-harm-homog-polyn-degree-3-H^1} needs to be satisfied, hence it holds
	\begin{align*}
		&(-b_{(1,2,0)}+2b_{(0,1,1)})\sigma_x^2+(b_{(1,2,0)}+2b_{(0,1,1)})\sigma_y^2=2b_{(0,1,1)}(\sigma_x^2+\sigma_y^2)\\
		&+b_{(1,2,0)}(\sigma_y^2-\sigma_x^2).
	\end{align*}
	In particular, because $\sigma_x^2$ and $\sigma_y^2$ assume the same values on subsets of $\partial B^{\mathbb{H}^1}_1(0)$ of the same measure, $b_{(1,2,0)}(\sigma_y^2-\sigma_x^2)$ gives null contribution in $a_{3,1}^{\mathbb{H}^1},$ thus we have
	\[
	\begin{split}
		&a_{3,1}^{\mathbb{H}^1}=\frac{1}{4}\int_{\partial B^{\mathbb{H}^1}_1(0)}\frac{T_{3,1}(\sigma)}{\sqrt{\sigma_x^2+\sigma_y^2}}
		\hspace{0.05cm}d\sigma_{\mathbb{H}^1}=\frac{1}{4}\int_{\partial B^{\mathbb{H}^1}_1(0)}\frac{2b_{(0,1,1)}(\sigma_x^2+\sigma_y^2)}{\sqrt{\sigma_x^2+\sigma_y^2}}
		\hspace{0.05cm}d\sigma_{\mathbb{H}^1}\\
		&=\frac{b_{(0,1,1)}}{2}\int_{\partial B^{\mathbb{H}^1}_1(0)}\sqrt{\sigma_x^2+\sigma_y^2}
		\hspace{0.1cm}d\sigma_{\mathbb{H}^1},
	\end{split}
	\]
	which yields
	\[a_{3,1}^{\mathbb{H}^1}=\frac{b_{(0,1,1)}}{2}\int_{\partial B^{\mathbb{H}^1}_1(0)}\sqrt{\sigma_x^2+\sigma_y^2}
	\hspace{0.1cm}d\sigma_{\mathbb{H}^1}.\]
    Taking now $u=x-3yt-2x^3,$ we then obtain $a_{3,1}^{\mathbb{H}^1}<0.$ Moreover, $a_{2}^{\mathbb{H}^1}=0$ according again to \eqref{defin-a_k-a_h,k-funct-J_u^H^1}. As a matter of fact, $Q_2$ is the square norm of the $\mathbb{H}^1$-gradient of the homogeneous polynomial of degree $2,$ $P_2,$ coming from the Taylor expansion of $u$ at $0,$ which is the null polynomial in case of $u=x-3yt-2x^3.$ Hence, we reach $a_{2}^{\mathbb{H}^1}+2a_{3,1}^{\mathbb{H}^1}<0.$
\end{proof}
	\section{Explicit computation of the counterexample}\label{explicit_computation}
	In this section we explicitly check that fixing the polynomial $$u=x-3yt-2x^3,$$ $\mathcal{I}_{u}^{\mathbb{H}^1}$ is monotone decreasing in a right neighborhood of $r=0.$
	\begin{lem}\label{lemma-expl-count}
		Let $u=x-3yt-2x^3.$ Then $\Delta_{\mathbb{H}^1}u=0$ and $\mathcal{I}_{u}^{\mathbb{H}^1}$ is monotone decreasing around $r=0.$
	\end{lem}
	\begin{proof}
		First, recalling that $u$ defining $\mathcal{I}_{u}^{\mathbb{H}^1}$ has to be ${\mathbb{H}^1}$-harmonic, we show that $\Delta_{\mathbb{H}^1}u=0.$ Directly computing, we get
	\begin{equation}\label{nabla-H-^1-x-3yt-2x^3}
				\begin{split}
			    &Xu=1 -6x^2-6y^2,\\
				&Yu=3(-t+2xy),
				\end{split}
	\end{equation}
	which implies
	\[\Delta_{\mathbb{H}^1}u=-12x+12x=0.\]
    Now, we focus on the behavior of $\mathcal{I}_{u}^{\mathbb{H}^1}.$ Substituting \eqref{nabla-H-^1-x-3yt-2x^3} in \eqref{rewriting-J_u^H^1-1}, it holds
		\begin{align*}
			&\mathcal{I}_{u}^{\mathbb{H}^1}(r)=\frac{1}{r^2}\int_{B^{\mathbb{H}^1}_r(0)}\frac{(1 -6(x^2+y^2))^2+9(-t+2xy)^2}{|\xi|_{\mathbb{H}^1}^{2}}\hspace{0.1cm}d\xi.
		\end{align*}
	Next, repeating the same steps exploited to achieve \eqref{rewriting-J_u^H^1-2}, we have
		\begin{align*}
			&\mathcal{I}_{u}^{\mathbb{H}^1}(r)=\frac{1}{r^2}\bigg(\int^r_0\rho\bigg(\int_{\partial B^{\mathbb{H}^1}_1(0)}\frac{(1 -6\rho^2(\sigma_x^2+\sigma_y^2))^2}{\sqrt{\sigma_x^2+\sigma_y^2}}\hspace{0.1cm}d\sigma_{\mathbb{H}^1}\bigg)d\rho\\
			&+\int^r_0\rho\bigg(\int_{\partial B^{\mathbb{H}^1}_1(0)}\frac{9\rho^4(-\sigma_t+2\sigma_x\sigma_y)^2}{\sqrt{\sigma_x^2+\sigma_y^2}}\hspace{0.1cm}d\sigma_{\mathbb{H}^1}\bigg)d\rho\bigg)\\
			&=\frac{1}{r^2}\bigg(\int^r_0\rho\bigg(\int_{\partial B^{\mathbb{H}^1}_1(0)}\frac{1 }{\sqrt{\sigma_x^2+\sigma_y^2}}\hspace{0.1cm}d\sigma_{\mathbb{H}^1}\bigg)d\rho\\
			&-\int^r_012\rho^3\bigg(\int_{\partial B^{\mathbb{H}^1}_1(0)}\sqrt{\sigma_x^2+\sigma_y^2}\hspace{0.1cm}d\sigma_{\mathbb{H}^1}\bigg)d\rho\\
			&+\int^r_0\rho^5\bigg(\int_{\partial B^{\mathbb{H}^1}_1(0)}\bigg(36(\sigma_x^2+\sigma_y^2)^{3/2}+\frac{9(-\sigma_t+2\sigma_x\sigma_y)^2}{\sqrt{\sigma_x^2+\sigma_y^2}}\bigg)\hspace{0.1cm}d\sigma_{\mathbb{H}^1}\bigg)d\rho\bigg)\\
			&=\frac{1}{r^2}\bigg(\frac{r^2}{2}\int_{\partial B^{\mathbb{H}^1}_1(0)}\frac{1 }{\sqrt{\sigma_x^2+\sigma_y^2}}\hspace{0.1cm}d\sigma_{\mathbb{H}^1}-3r^4\int_{\partial B^{\mathbb{H}^1}_1(0)}\sqrt{\sigma_x^2+\sigma_y^2}\hspace{0.1cm}d\sigma_{\mathbb{H}^1}\\
			&+\frac{r^6}{6}\int_{\partial B^{\mathbb{H}^1}_1(0)}\bigg(36(\sigma_x^2+\sigma_y^2)^{3/2}+\frac{9(-\sigma_t+2\sigma_x\sigma_y)^2}{\sqrt{\sigma_x^2+\sigma_y^2}}\bigg)\hspace{0.1cm}d\sigma_{\mathbb{H}^1}\bigg),
		\end{align*}
		which yields
		\begin{equation}\label{express-J-u-H^1-counter}
			\mathcal{I}_{u}^{\mathbb{H}^1}(r)=a_1^{\mathbb{H}^1}-2a_{3,1}^{\mathbb{H}^1}r^2+a_3^{\mathbb{H}^1}r^4,
		\end{equation}
		with
		\begin{align*}
			&a_1^{\mathbb{H}^1}=\frac{1}{2}\int_{\partial B^{\mathbb{H}^1}_1(0)}\frac{1 }{\sqrt{\sigma_x^2+\sigma_y^2}}\hspace{0.1cm}d\sigma_{\mathbb{H}^1},\\
			&2a_{3,1}^{\mathbb{H}^1}= 3\int_{\partial B^{\mathbb{H}^1}_1(0)}\sqrt{\sigma_x^2+\sigma_y^2}\hspace{0.1cm}d\sigma_{\mathbb{H}^1},\\
			&a_3^{\mathbb{H}^1}= \int_{\partial B^{\mathbb{H}^1}_1(0)}\bigg(36(\sigma_x^2+\sigma_y^2)^{3/2}+\frac{9(-\sigma_t+2\sigma_x\sigma_y)^2}{\sqrt{\sigma_x^2+\sigma_y^2}}\bigg)\hspace{0.1cm}d\sigma_{\mathbb{H}^1}.
		\end{align*}
		Explicitly calculating the derivative of $\mathcal{I}_{u}^{\mathbb{H}^1}$ as in \eqref{express-J-u-H^1-counter} and letting $r\to 0,$ we reach the thesis since $a_{3,1}^{\mathbb{H}^1}$ is positive by definition. 
	\end{proof}
	\section{Nonexistence of an Alt-Caffarelli-Friedman type monotonicity formula in $\mathbb{H}^1$}\label{decreasing_formula}
	In this section, we show that an Alt-Caffarelli-Friedman type monotonicity formula in $\mathbb{H}^1$ does not hold, at least considering the function \eqref{moneqH}.
	In fact, we can exploit  the counterexample to the increasing monotonicity of $\mathcal{I}_{u}^{\mathbb{H}^1}$ provided in Lemma \ref{lemma-expl-count}. 
	\begin{proof}[Proof of Theorem \ref{maintheorem}]
		We first note that
		\begin{equation}\label{equal-J_4,H^1-in-terms-J-u+^H^1,J_u-_H^1}
			J_u^{\mathbb{H}^1}(r)=\mathcal{I}_{u^+}^{\mathbb{H}^1}(r)\mathcal{I}_{u^-}^{\mathbb{H}^1}(r). 
		\end{equation}
		So, since $\mathcal{I}_{u^+}^{\mathbb{H}^1}$ and $\mathcal{I}_{u^-}^{\mathbb{H}^1}$ are nonnegative, we reach the desired result if we prove that they are both monotone decreasing.\\
		We claim that $\mathcal{I}_{u^+}^{\mathbb{H}^1}(r)=\mathcal{I}_{u^-}^{\mathbb{H}^1}(r).$ Before proving this, we show that the claim implies the monotone decreasing behavior of $\mathcal{I}_{u^+}^{\mathbb{H}^1}(r)$ and $\mathcal{I}_{u^-}^{\mathbb{H}^1}(r).$ 
		
		Indeed,  by the obvious fact  that $\mathcal{I}^{\mathbb{H}^1}_u(r)=\mathcal{I}_{u^+}^{\mathbb{H}^1}(r)+\mathcal{I}_{u^-}^{\mathbb{H}^1}(r),$ we deduce 
		$\mathcal{I}^{\mathbb{H}^1}_u(r)\allowbreak=2\mathcal{I}_{u^+}^{\mathbb{H}^1}(r)=2\mathcal{I}_{u^-}^{\mathbb{H}^1}(r),$ which immediately gives the decreasing monotonicity of $\mathcal{I}_{u^+}^{\mathbb{H}^1}$ and $\mathcal{I}_{u^-}^{\mathbb{H}^1}$ from Lemma \ref{lemma-expl-count}. 
		
		As a byproduct of this remark, in this way, we prove that
		$J_u^{\mathbb{H}^1}(r)$ is monotone decreasing, because it is the product of two positive monotone decreasing functions, see \eqref{equal-J_4,H^1-in-terms-J-u+^H^1,J_u-_H^1}.
		
		It remains to show the claim holds. To this end, we write
		\begin{align*}
			\mathcal{I}_{u^+}^{\mathbb{H}^1}(r)&=\frac{1}{r^2}\int_{B^{\mathbb{H}^1}_r(0)\cap\{u>0\}}\frac{(1 -6(x^2+y^2))^2+9(-t+2xy)^2}{|\xi|_{\mathbb{H}^1}^{2}}\hspace{0.1cm}d\xi\\
			&=\frac{1}{r^2}\int_{B^{\mathbb{H}^1}_r(0)\cap\{x-3yt-2x^3>0\}}\frac{(1 -6(x^2+y^2))^2+9(-t+2xy)^2}{|\xi|_{\mathbb{H}^1}^{2}}\hspace{0.1cm}d\xi,
		\end{align*}
		and we apply the change of variables
		\begin{equation}\label{changeofvariables}\xi=(x,y,t)=T(\eta)=T(w,z,s)\coloneqq (-w,-z,s),\end{equation}
		which yields
		\begin{align*}
			&\mathcal{I}_{u^+}^{\mathbb{H}^1}(r)=\frac{1}{r^2}\int_{B^{\mathbb{H}^1}_r(0)\cap\{-w-3(-z)s-2(-w)^3>0\}}\frac{(1 -6(w^2+z^2))^2+9(-s+2wz)^2}{|\eta|_{\mathbb{H}^1}^{2}}\hspace{0.1cm}d\eta\\
			&=\frac{1}{r^2}\int_{B^{\mathbb{H}^1}_r(0)\cap\{w-3zs-2w^3<0\}}\frac{(1 -6(w^2+z^2))^2+9(-s+2wz)^2}{|\eta|_{\mathbb{H}^1}^{2}}\hspace{0.1cm}d\eta=J_{u^-}^{\mathbb{H}^1}(r),
		\end{align*}
		and thus the claim follows.
	\end{proof}
	\section{A further generalization with application}\label{gen_app}
	Let us consider now the following two phase continuous function
	$$
	u_{\alpha_1,\alpha_2}=\alpha_1 u^+-\alpha_2 u^-
	$$
	in $\mathbb{H}^1,$
	where, as usual, $u(x,y,t)=x-3yt-2x^3.$ Then,
	we first note that
	\begin{equation}\label{equal-J_4,H^1-in-terms-J-u+^H^1,J_u-_H^1gen}
		J_{ u_{\alpha_1,\alpha_2}}^{\mathbb{H}^1}(r)=\mathcal{I}_{\alpha_1 u^+}^{\mathbb{H}^1}(r)\mathcal{I}_{\alpha_2 u^-}^{\mathbb{H}^1}(r). 
	\end{equation}
	Since the zero level of $ u_{\alpha_1,\alpha_2}$ is the same of $u,$ we remark that
	\begin{align*}
		\mathcal{I}_{\alpha_1u^+}^{\mathbb{H}^1}(r) &=\frac{\alpha_1^2}{r^2}\int_{B^{\mathbb{H}^1}_r(0)\cap\{x-3yt-2x^3>0\}}\frac{(1 -6(x^2+y^2))^2+9(-t+2xy)^2}{|\xi|_{\mathbb{H}^1}^{2}}\hspace{0.1cm}d\xi,
	\end{align*}
	as well as
	\begin{align*}
		\mathcal{I}_{\alpha_2u^-}^{\mathbb{H}^1}(r) &=\frac{\alpha_1^2}{r^2}\int_{B^{\mathbb{H}^1}_r(0)\cap\{x-3yt-2x^3<0\}}\frac{(1 -6(x^2+y^2))^2+9(-t+2xy)^2}{|\xi|_{\mathbb{H}^1}^{2}}\hspace{0.1cm}d\xi.
	\end{align*}
	Hence, performing the same change of variables introduced in \eqref{changeofvariables}, we obtain \begin{equation}\label{relationship}
		\mathcal{I}_{\alpha_1u^+}^{\mathbb{H}^1}(r)=\frac{\alpha_1^2}{\alpha_2^2}\mathcal{I}_{\alpha_2u^-}^{\mathbb{H}^1}(r).
	\end{equation}
	On the other hand, in this case, keeping in mind \eqref{relationship} it results
	\begin{equation}\label{equal-J_4,H^1-in-terms-J-u+^H^1,J_u-_H^1gen2}\begin{split}
			\mathcal{I}_{u_{\alpha_1,\alpha_2}}^{\mathbb{H}^1}(r)&=\mathcal{I}_{\alpha_1 u^+}^{\mathbb{H}^1}(r)+\mathcal{I}_{\alpha_2u^-}^{\mathbb{H}^1}(r)=\left(\frac{\alpha_1^2}{\alpha_2^2}+1\right)\mathcal{I}_{\alpha_2 u^-}^{\mathbb{H}^1}(r)\\
			&=\left(\alpha_1^2+\alpha_2^2\right)\mathcal{I}_{u^-}^{\mathbb{H}^1}(r)=\frac{\alpha_1^2+\alpha_2^2}{2}\mathcal{I}_{u}^{\mathbb{H}^1}(r).\end{split}
	\end{equation}
	
	As a consequence of Lemma \ref{lemma-expl-count}, from \eqref{equal-J_4,H^1-in-terms-J-u+^H^1,J_u-_H^1gen2} follows that
	$$
	\mathcal{I}_{u_{\alpha_1,\alpha_2}}^{\mathbb{H}^1},\quad \mathcal{I}_{\alpha_2 u^-}^{\mathbb{H}^1},\quad \mathcal{I}_{\alpha_1 u^+}^{\mathbb{H}^1}
	$$   
	are monotone decreasing in a right neighborhood of $0$.  Hence, we conclude  from \eqref{equal-J_4,H^1-in-terms-J-u+^H^1,J_u-_H^1gen} that $J_{ u_{\alpha_1,\alpha_2}}^{\mathbb{H}^1}$ is monotone decreasing because it is the product of the two positive monotone decreasing functions  $\mathcal{I}_{\alpha_2 u^-}^{\mathbb{H}^1}$ and $ \mathcal{I}_{\alpha_1 u^+}^{\mathbb{H}^1}.$ 
	
	We conclude the paper with an application coming from the two phase function $u_{\alpha_1,\alpha_2}.$ 
	
	The construction of non-trivial solutions of two-phase free boundary problems is not easy even in the Euclidean setting. 
	
	Nevertheless, we remark that the function $u_{\alpha_1,\alpha_2}$ is solution of the following two phase problem in the Heisenberg group,  
	\begin{equation}\label{two_phase_Heisenberg_2final}
		\begin{cases}
			\Delta_{\mathbb{H}^1} u=0& \mbox{in }\Omega^+(u):= \{x\in \Omega:\hspace{0.1cm} u(x)>0\},\\
			\Delta_{\mathbb{H}^1} u=0& \mbox{in }\Omega^-(u):=\mbox{Int}(\{x\in \Omega:\hspace{0.1cm} u(x)\leq 0\}),\\
			|\nabla_{\mathbb{H}^1} u^+|^2-|\nabla_{\mathbb{H}^1} u^-|^2= g_{\alpha_1,\alpha_2}&\mbox{on }\mathcal{F}(u):=\partial \Omega^+(u)\cap \Omega,
		\end{cases}
	\end{equation}
	where $$g_{\alpha_1,\alpha_2}(x,y,t)=(\alpha_1^2-\alpha_2^2)\left((1 -6(x^2+y^2))^2+9(-t+2xy)^2\right).$$
	
	As a consequence, supposing that $g_{\alpha_1,\alpha_2}$ is a given function, with $\alpha_1,\alpha_2>0$  fixed numbers chosen in such a way that $\alpha_1^2-\alpha_2^2>0.$ Then $u_{\alpha_1,\alpha_2}$ is a solution of the two-phase free boundary problem 
	\begin{equation}\label{two_phase_Heisenberg_2final}
		\begin{cases}
			\Delta_{\mathbb{H}^1} u=0& \mbox{in }\Omega^+(u):= \{x\in \Omega:\hspace{0.1cm} u(x)>0\},\\
			\Delta_{\mathbb{H}^1} u=0& \mbox{in }\Omega^-(u):=\mbox{Int}(\{x\in \Omega:\hspace{0.1cm} u(x)\leq 0\}),\\
			|\nabla_{\mathbb{H}^1} u^+|^2-|\nabla_{\mathbb{H}^1} u^-|^2= g&\mbox{on }\mathcal{F}(u):=\partial \Omega^+(u)\cap \Omega.
		\end{cases}
	\end{equation}
	\bibliographystyle{abbrv}
	\bibliography{biblio}
\end{document}